\documentclass[11pt]{amsart}
\usepackage{amssymb, verbatim}
\usepackage[all]{xy}

\title{Restricted volumes on K\"ahler manifolds}
\author[T.C. Collins]{Tristan C. Collins}
\address{Department of Mathematics, Massachusetts Institute of Technology, 77 Massachusetts Avenue, Cambridge, MA 02139}
\email{tristanc@mit.edu}
\author[V. Tosatti]{Valentino Tosatti}
\address{Department of Mathematics and Statistics, McGill University, Montr\'eal, Qu\'ebec H3A 0B9, Canada}
\address{Department of Mathematics, Northwestern University, 2033 Sheridan Road, Evanston, IL 60208}
\email{valentino.tosatti@mcgill.ca}
\thanks{The first-named author is supported in part by NSF grant DMS-1506652, the second-named author by a Sloan Research Fellowship and by NSF grants DMS-1308988 and DMS-1903147.}
\dedicatory{Dedicated to Professor Ahmed Zeriahi on the occasion of his retirement.}

\theoremstyle{plain}
\newtheorem{thm}{Theorem}[section]
\newtheorem{prop}[thm]{Proposition}

\newtheorem{lem}[thm]{Lemma}
\newtheorem{cor}[thm]{Corollary}
\newtheorem{conj}[thm]{Conjecture}

\theoremstyle{definition}
\newtheorem{rem}[thm]{Remark}
\theoremstyle{definition}
\newtheorem{ex}[thm]{Example}

\numberwithin{equation}{section}
\newcommand{\ddbar}{\sqrt{-1} \partial \overline{\partial}}

\newcommand{\Null}{\textnormal{Null}}

\newcommand{\ve}{\varepsilon}
\newcommand{\vp}{\varphi}
\newcommand{\ti}[1]{\tilde{#1}}

\renewcommand{\leq}{\leqslant}
\renewcommand{\geq}{\geqslant}

\newcommand{\vol}{\mathrm{Vol}}
\newcommand{\ov}[1]{\overline{#1}}

\begin{document}
\begin{abstract}
We study numerical restricted volumes of $(1,1)$ classes on compact K\"ahler manifolds, as introduced by Boucksom. Inspired by work of Ein-Lazarsfeld-Musta\c{t}\u{a}-Nakamaye-Popa on restricted volumes of line bundles on projective manifolds, we pose a natural conjecture to the effect that irreducible components of the non-K\"ahler locus of a big class should have vanishing numerical restricted volume. We prove this conjecture when the class has a Zariski decomposition, and give several applications.
\end{abstract}
\maketitle
\section{Introduction}
Let $X$ be a smooth projective variety over $\mathbb{C}$ and $L\to X$ a holomorphic line bundle. Given $V\subset X$ an irreducible positive-dimensional analytic subvariety, we define (following \cite{ELMNP}, and originating in \cite{HM,Ta}) the restricted volume of $L$ along $V$ to be
$$\vol_{X|V}(L)=\limsup_{m\to\infty}\frac{\dim\mathrm{Im}(H^0(X,mL)\to H^0(V,mL|_V))}{m^{\dim V}/(\dim V)!}.$$
In particular, $\vol_{X|V}(L)>0$ implies that $L|_V$ is big, but not conversely. Geometrically, $\vol_{X|V}(L)>0$ means that for all $m$ sufficiently large and divisible, the rational map
$\Phi_m:X\dashrightarrow Y_m\subset\mathbb{CP}^{N_m}$ given by sections of $mL$ restricts to a map $\Phi_m|_V:V\dashrightarrow W_m\subset\mathbb{CP}^{N_m}$ which is birational with its image $W_m$ (see \cite[Corollary 2.5]{BCL}).

In the case when $V=X$ we obtain the familiar volume of $L$, denoted by $\vol_{X}(L)$ or just $\vol(L)$. It is easy to extend the definition of restricted volumes to $\mathbb{Q}$-divisors by homogeneity. Restricted volumes have proved to be an extremely useful tool in algebraic geometry, see e.g. \cite{BCL, BFJ, CN,CHPW,Den, DPa,LM,Lo,Le,Nak2,PT} and references therein.

Our main interest is in developing the theory of restricted volumes in the transcendental setting when $X$ is a compact K\"ahler manifold and $L$ is replaced by a $(1,1)$ cohomology class $[\alpha]$. In the case when $V\not\subset E_{\rm nK}(\alpha)$, the non-K\"ahler locus of $[\alpha]$ as defined in \cite{Bo2}, Hisamoto \cite{Hi} and Matsumura \cite{Ma} independently found an analytic formula for the restricted volume, generalizing Boucksom's analytic formula for the volume of a line bundle \cite{Bo}, which however is poorly behaved when $V$ is contained in $E_{\rm nK}(\alpha)$.

On the other hand, in the absolute case when $V=X$, a general theory of moving intersection products was developed by Boucksom and his collaborators \cite{BoT, BDPP, BEGZ,BFJ}. Our goal is to extend these ideas to the relative case, and to develop a parallel theory to the algebraic case \cite{ELMNP}. In the case of nef classes, this was achieved in our earlier work \cite{CT} (see also \cite{CT2,CT3}).

Let now $X$ be a compact K\"ahler manifold, let $\alpha$ a closed real $(1,1)$ form such that its cohomology class $[\alpha]$ is pseudoeffective, and let $V\subset X$ an irreducible positive-dimensional analytic subvariety. We wish to define the numerical restricted volume $\langle \alpha^{\dim V}\rangle_{X|V}$. If $V\subset E_{\rm nn}(\alpha)$ (the non-nef locus of $[\alpha]$ as defined in \cite{Bo2}) we define $\langle \alpha^{\dim V}\rangle_{X|V}=0$, while if $V\not\subset E_{\rm nn}(\alpha)$ we define
$$\langle \alpha^{\dim V}\rangle_{X|V}=\lim_{\ve\downarrow 0}\sup_{T}\int_{V_{\rm reg}} ((T+\ve\omega)|_{V_{\rm reg}})_{\rm ac}^{\dim V},$$
where $(\cdot)_{\rm ac}$ denoted the absolutely continuous part in the Lebesgue decomposition (as in \cite{Bo}), and the supremum is over all closed real $(1,1)$ currents $T$ in the class $[\alpha]$ which satisfy $T\geq -\ve\omega$ on $X$, such that $T$ has analytic singularities which do not contain $V$. The fact that such currents $T$ exist follows from the assumption that $V\not\subset E_{\rm nn}(\alpha)$, as we will see. Also, if we choose $V=X$ then we have
$$\langle\alpha^n\rangle_{X|X}=\vol(\alpha),$$
the volume of $[\alpha]$, as defined by Boucksom \cite{Bo}.

This analytic definition coincides with the ``mobile intersection number''
$$(\alpha^{\dim V}\cdot [V])_{\geq 0},$$
defined in Boucksom's thesis \cite[Definition 3.2.1]{BoT}, where he proves some of its basic properties. For example, \cite[Proposition 3.2.2]{BoT} shows that when the class $[\alpha]$ is nef then we simply have that
$$\langle \alpha^{\dim V}\rangle_{X|V}=\int_V\alpha^{\dim V}.$$

For a general pseudoeffective class $[\alpha]$ we define its null locus to be
$$\Null(\alpha)=\bigcup_{\langle \alpha^{\dim V}\rangle_{X|V}=0} V,$$
where the union is over all irreducible analytic subvarieties of $X$ with vanishing numerical restricted volume. When $[\alpha]$ is also nef this coincides with the usual null locus (cf. \cite{CT}).

\begin{conj}\label{conjec}
Let $X$ be a compact K\"ahler manifold and $[\alpha]$ a pseudoeffective $(1,1)$ class. Then we have
\begin{equation}\label{enk}
E_{\rm nK}(\alpha)=\Null(\alpha).
\end{equation}
Equivalently, if $V$ is an irreducible component of $E_{\rm nK}(\alpha)$ then we have
\begin{equation}\label{vol0}
\langle \alpha^{\dim V}\rangle_{X|V}=0.
\end{equation}
\end{conj}

Clearly it is enough to show \eqref{enk} when $[\alpha]$ is big.
In the case when $[\alpha]$ is nef this is the main theorem of our previous work \cite{CT} (which was itself a transcendental extension of Nakamaye's Theorem \cite{Nak}), and so Conjecture \ref{conjec} is the natural extension of our earlier result to all pseudoeffective $(1,1)$ classes. As we will explain presently, it is also the natural transcendental extension of a later result of Namakaye \cite{Nak2} and Ein-Lazarsfeld-Musta\c{t}\u{a}-Nakamaye-Popa \cite{ELMNP}.

Recall that every pseudoeffective class $[\alpha]$ has a divisorial Zariski decomposition $[\alpha]=P+N$, introduced by Boucksom \cite{Bo2} (and Nakayama \cite{Na} in the algebraic case). The class $P$ in general is nef in codimension $1$, in the sense that each irreducible component of $E_{\rm nn}(P)$ has codimension at least $2$ in $X$. If $P$ is actually nef, then we say that $[\alpha]$ has a Zariski decomposition. Thanks to a classical result of Zariski \cite{Za}, every pseudoeffective $(1,1)$ class has a Zariski decomposition when $\dim X=2$, however there are examples in dimensions greater than $2$ of big classes which do not have a Zariski decomposition, not even on any birational model \cite{Bo2, Na}.

\begin{thm}\label{l5}
Conjecture \ref{conjec} holds if the class $[\alpha]$ has a Zariski decomposition. In particular, it always holds when $X$ is a complex surface.
\end{thm}

In fact, it is enough to assume that $[\alpha]$ admit a Zariski decomposition on some bimeromorphic model, see Theorem \ref{l7}.

The connection with the algebraic setting is given by the following:
\begin{thm}\label{regg}
If $X$ is a projective manifold, $L$ is a pseudoeffective line bundle and $V\subset X$ is a positive-dimensional irreducible analytic subvariety, then given any ample line bundle $H$ on $X$ we have
\begin{equation}\label{regular}
\langle c_1(L)^{\dim V}\rangle_{X|V}=\lim_{\ve\downarrow 0} \vol_{X|V}(L+\ve H),
\end{equation}
where we restrict ourselves to $\ve\in\mathbb{Q}_{>0}$. In particular, we always have
\begin{equation}\label{oneway}
\vol_{X|V}(L)\leq \langle c_1(L)^{\dim V}\rangle_{X|V}.
\end{equation}
\end{thm}
By homogeneity, this result also holds when $L$ and $H$ are just $\mathbb{Q}$-divisors.
The quantity on the right hand side of \eqref{regular} has been recently considered explicitly in \cite{CHPW}. In fact, more is true. When $X$ is projective and $V\not\subset \mathbb{B}_+(L)=E_{\rm nK}(c_1(L))$ (see e.g. \cite{ELMNP2} for more on $\mathbb{B}_+$ and $\mathbb{B}_-$, and \cite[Theorem 2.3]{To3} for this equality) then we simply have
$$\langle c_1(L)^{\dim V}\rangle_{X|V}=\vol_{X|V}(L),$$
thanks to \cite{Hi,Ma} (see Theorem \ref{hisamat}). When $V\subset \mathbb{B}_-(L)=E_{\rm nn}(c_1(L))$ we have
$$\langle c_1(L)^{\dim V}\rangle_{X|V}=\vol_{X|V}(L)=0.$$
This is because $\langle c_1(L)^{\dim V}\rangle_{X|V}=0$ by definition, while $\vol_{X|V}(L)=0$ since $\mathbb{B}_-(L)\subset\mathbb{B}(L)$ (the stable base locus of $L$) and it is clear from the definition that $\vol_{X|V}(L)=0$ holds whenever $V\subset\mathbb{B}(L)$ (this observation also justifies our definition of $\langle \alpha^{\dim V}\rangle_{X|V}=0$ when $V\subset E_{\rm nn}(\alpha)$).

But when $V\subset \mathbb{B}_+(L)$ in general we have that
$$\langle c_1(L)^{\dim V}\rangle_{X|V}\neq\vol_{X|V}(L),$$
see Examples \ref{exx} and \ref{exx2} (or \cite[Example 2.3]{CHPW} for another example). Note that this example shows that $\vol_{X|V}(L)$ is {\em not} in general a numerical invariant of the line bundle $L$, while obviously $\langle c_1(L)^{\dim V}\rangle_{X|V}$ is.

It is also important to note that given $V\subset X$, the map $[\alpha]\mapsto \langle \alpha^{\dim V}\rangle_{X|V}$ is not continuous in general on the big cone (unlike the usual volume function), see Example \ref{exx2}.

In the algebraic setting, we show that Conjecture \ref{conjec} holds in certain important cases thanks to existing work in the literature.
Using Theorem \ref{regg} together with the main result of Ein-Lazarsfeld-Musta\c{t}\u{a}-Nakamaye-Popa \cite[Theorem 5.7]{ELMNP} we show that:
\begin{prop}\label{prob1}
Conjecture \ref{conjec} holds when $X$ is projective and $[\alpha]=c_1(D)$ for $D$ a big $\mathbb{R}$-divisor.
\end{prop}
In fact, we also show in Proposition \ref{prob2} that conversely the validity of Conjecture \ref{conjec} would give a transcendental proof of \cite[Theorem 5.7]{ELMNP}.

Following Boucksom-Favre-Jonsson \cite{BFJ} in the algebraic case, we observe that a recent result of Witt-Nystr\"om \cite{WN} implies:
\begin{prop}\label{ortorv}
Conjecture \ref{conjec} (in its equivalent form \eqref{vol0}) holds when $X$ is projective, $[\alpha]$ is any big $(1,1)$ class, and $V$ is a prime divisor.
\end{prop}
We now discuss a few applications of Conjecture \ref{conjec}.
As in \cite{ELMNP, Nak2}, we draw a connection between restricted volumes and moving Seshadri constants. These are a generalization, first introduced in \cite{Nak2}, of the usual Seshadri constants, which are defined for nef classes, to arbitrary $(1,1)$ classes. Given a $(1,1)$ class $[\alpha]$, we define the moving Seshadri constant $\ve(\|\alpha\|,x)$ as follows: if $x\in E_{\rm nK}(\alpha)$ we set $\ve(\|\alpha\|,x)=0$, and otherwise we set
$$\ve(\|\alpha\|,x)=\sup_{\mu^*[\alpha]=[\beta]+[E]}\ve(\beta,\mu^{-1}(x)),$$
where the supremum is over all modifications $\mu:\ti{X}\to X$, which are isomorphisms near $x$, and over all decompositions
$\mu^*[\alpha]=[\beta]+[E]$ where $[\beta]$ is a K\"ahler class and $E$ is an effective $\mathbb{R}$-divisor which does not contain $\mu^{-1}(x)$, and where $\ve$ denotes the usual Seshadri constant of a nef class. Clearly the moving Seshadri constants are trivial if $[\alpha]$ is not big. This definition agrees with the algebraic definition in \cite{ELMNP, Nak2} when $X$ is projective and $[\alpha]=c_1(D)$
for a big $\mathbb{R}$-divisor $D$.
For a nef class $[\alpha],$ the Seshadri constants are also given by the formula
\begin{equation}\label{subvar}
\ve(\alpha,x)=\inf_{V\ni x}\left(\frac{\int_V\alpha^{\dim V}}{\mathrm{mult}_x V}\right)^{\frac{1}{\dim V}},
\end{equation}
which is proved in \cite[Theorem 2.8]{To} (cf. \cite[Proposition 5.1.9]{Laz} in the algebraic case), and the infimum is achieved when $[\alpha]$ is K\"ahler.
The following is a generalization of this result to moving Seshadri constants, in analogy with \cite[Proposition 6.7]{ELMNP}:

\begin{thm}\label{movable2} Assume Conjecture \ref{conjec}.
Given any pseudoeffective class $[\alpha]$ and any $x\in X$ we have
$$\ve(\|\alpha\|,x)=\inf_{V\ni x}\left(\frac{\langle \alpha^{\dim V}\rangle_{X|V}}{\mathrm{mult}_x V}\right)^{\frac{1}{\dim V}},$$
where the infimum is over all irreducible positive-dimensional analytic subvarieties $V$ containing $x$.
\end{thm}

The following result reproves and extends the main result of \cite{Nak2}.
\begin{thm}\label{contt}
Assume Conjecture \ref{conjec}.
Given any $x\in X$, the map $[\alpha]\mapsto \ve(\|\alpha\|,x)$ is continuous as $[\alpha]$ varies in $H^{1,1}(X,\mathbb{R})$.
\end{thm}
We also show that Conjecture \ref{conjec} implies a positive answer to a question of Boucksom \cite{BoT} about a slightly different generalization of Seshadri constants to pseudoeffective classes, which turns out to agree with the moving Seshadri constants, see Theorem \ref{bouses}.

Our last application is a generalization of the local ampleness criterion of Takayama \cite[Proposition 2.1]{Ta2}:
\begin{thm}\label{localamp}
Let $X$ be a compact K\"ahler manifold and $[\alpha]$ a big class. Let $T\geq 0$ be a closed positive current in the class $[\alpha]$ which is a smooth K\"ahler metric on a nonempty open set $U\subset X$. Then, assuming Conjecture \ref{conjec}, we have that $U\subset E_{\rm nK}(\alpha)^c$.
\end{thm}

Here and in the rest of the paper, we use interchangably ``K\"ahler metric'' and ``K\"ahler form''. Of course, if $[\alpha]$ is also nef then we do not need to assume Conjecture \ref{conjec}, thanks to our earlier work \cite{CT}. Such ampleness criteria (in the projective case) were used to obtain quasi-projectivity criteria in \cite[Theorem 6.1]{LWX} and \cite[Theorem 6]{ST}.\\

This paper is organized as follows. After defining numerical restricted volumes and proving some of their basic properties in Section \ref{sect2}, where we also prove Theorem \ref{regg}, we state our main conjecture in Section \ref{sect3}, and we give the proofs of Proposition \ref{prob1} and Theorem \ref{localamp}. In Section \ref{sect4} we prove several technical results, and in Section \ref{sect5} we prove a Fujita type approximation result. In Section \ref{sect6} we prove Theorem \ref{l5}, and in Section \ref{sect7} we prove Theorems \ref{movable2} and \ref{contt}, while Proposition \ref{ortorv} is proved in Section \ref{sect8}.\\

{\bf Acknowledgments. }It is a pleasure to dedicate this paper to Professor Ahmed Zeriahi, whose contributions to the study of degenerate complex Monge-Amp\`ere equations has had a tremendous impact on our work. We are also grateful to S. Boucksom, J.-P. Demailly, R. Lazarsfeld, M. P\u{a}un, M. Popa and Y.-T. Siu for many enlightening discussions about these topics, to L. Ein, B. Lehmann, J. Lesieutre, M. Musta\c{t}\u{a} and D. Witt-Nystr\"om for useful communications, and to the referee for detailed comments that improved the exposition. Part of this work was carried out during the second-named author's visits at the Center for Mathematical Sciences and Applications at Harvard University and at the Yau Mathematical Sciences Center at Tsinghua University in Beijing, which he would like to thank for the hospitality.

\section{Numerical restricted volumes of $(1,1)$ classes}\label{sect2}
Let $(X,\omega)$ be a compact K\"ahler manifold, $[\alpha]$ a pseudoeffective real $(1,1)$ class on $X$. First, let us recall some standard concepts, referring for example to \cite{Bo2} for details. Every pseudoeffective class contains closed positive $(1,1)$-currents, which are of the form $T=\alpha+\ddbar\vp$, where $\alpha$ is a closed real $(1,1)$-form representing the class and $\vp$ is a quasi-psh function such that $T\geq 0$ in the weak sense. Furthermore, one can always find such a current $T_{min}=\alpha+\ddbar\vp_{min}$ which has minimal singularities, in the sense that for every other closed positive current $S=\alpha+\ddbar\psi\geq 0$ in the class $[\alpha]$ there is some constant $C$ such that $\psi\leq \vp_{min}+C$ holds on $X$. In general $T_{min}$ is not unique, but by definition the difference of the potentials of two such currents is bounded.

We will also say that a closed real $(1,1)$-current $T=\alpha+\ddbar\vp$ is almost positive if it satisfies $T\geq -C\omega$ weakly for some $C\geq 0$.
Such a current is said to have analytic singularities if there are a coherent ideal sheaf $\mathcal{I}\subset\mathcal{O}_X$ and $c\in\mathbb{R}_{>0}$ such that $\vp$ is locally equal to $$c\log\sum_{j=1}^N|f_j|^2+h,$$
where the $f_j$ are local generators of $\mathcal{I}$ and $h$ is smooth. Such a current $T$ is then smooth outside the closed analytic subvariety $V(\mathcal{I})$.

A K\"ahler current $T$ is a closed positive $(1,1)$-current which satisfies $T\geq\ve\omega$ for some $\ve>0$, and a class $[\alpha]$ is called big if it contains a K\"ahler current. By Demailly's regularization \cite{Dem92}, big classes contain K\"ahler currents with analytic singularities. The intersection of the singular sets of all K\"ahler currents in a big class $[\alpha]$ is the non-K\"ahler locus $E_{\rm nK}(\alpha)\subset X$. This subset is a closed proper analytic subvariety, since it is shown in \cite{Bo} that $[\alpha]$ contains a K\"ahler current with analytic singularities precisely along $E_{\rm nK}(\alpha)$. Furthermore, $E_{\rm nK}(\alpha)$ is empty iff $[\alpha]$ is K\"ahler. We also define $E_{\rm nK}(\alpha)=X$ when $[\alpha]$ is pseudoeffective but not big.

For a pseudoeffective class $[\alpha]$, the non-nef locus $E_{\rm nn}(\alpha)\subset X$ is then defined by
\begin{equation}\label{moron}
E_{\rm nn}(\alpha)=\bigcup_{\ve>0}E_{\rm nK}(\alpha+\ve\omega),
\end{equation}
which is easily seen to be independent of the choice of $\omega$. The non-nef locus is in general a countable union of closed analytic subvarieties \cite{Les}, and it is empty iff $[\alpha]$ is nef.

We are now given $V\subset X$ an irreducible $k$-dimensional analytic subvariety, $k>0$. Following the terminology of \cite{Le} in the algebraic setting, we say that a pseudoeffective $(1,1)$-class $[\alpha]$ is $V$-pseudoeffective ($V$-psef) if $V\not\subset E_{\rm nn}(\alpha)$, and that $[\alpha]$ is $V$-big if $V\not\subset E_{\rm nK}(\alpha)$ (hence $[\alpha]$ is big). It is clear that $[\alpha]$ is $V$-psef iff $[\alpha+\ve\omega]$ is $V$-big for all $\ve>0$.
Furthermore, a pseudoeffective $(1,1)$-class $[\alpha]$ is $V$-big iff it contains a K\"ahler current with analytic singularities which is smooth at the generic point of $V$ (thanks to the above-mentioned result of \cite{Bo}). This observation is the analog of \cite[Lemma 2.8]{Le}.  It then follows that $[\alpha]$ is $V$-psef iff for every $\ve>0$ it contains a closed real $(1,1)$ current $T_\ve=\alpha+\ddbar\vp_\ve\geq -\ve\omega$ with analytic singularities which is smooth at the generic point of $V$.

We then wish to define the numerical restricted volume $\langle \alpha^{k}\rangle_{X|V}$ of $[\alpha]$ on $V$. As we said in the introduction, if $V\subset E_{\rm nn}(\alpha)$ we simply define $\langle \alpha^{k}\rangle_{X|V}=0$. If on the other hand $V\not\subset E_{\rm nn}(\alpha)$ (i.e. $[\alpha]$ is $V$-psef)
then we define
\begin{equation}\label{voldef}
\langle \alpha^{k}\rangle_{X|V}=\lim_{\ve\downarrow 0}\sup_{T}\int_{V_{\rm reg}} ((T+\ve\omega)|_{V_{\rm reg}})_{\rm ac}^{k},
\end{equation}
where the supremum is over all closed real $(1,1)$ currents $T$ in the class $[\alpha]$ which satisfy $T\geq -\ve\omega$ on $X$, such that $T$ has analytic singularities which do not contain $V$ (and as we just said such currents exist precisely when $[\alpha]$ is $V$-psef).
Note that the limit as $\ve\to 0$ exists because the map
$$\ve\mapsto \sup_{T}\int_{V_{\rm reg}} ((T+\ve\omega)|_{V_{\rm reg}})_{\rm ac}^{k},$$
is monotone increasing in $\ve$, and it is also not hard to see that it is independent of the choice of $\omega$.
 The number $\langle \alpha^{k}\rangle_{X|V}$ is finite thanks to the following lemma, which is a combination of \cite[Proposition 2.6]{Bo} and \cite[proposition 4.1]{Ma}.
\begin{lem}\label{finite}
Let $(X,\omega)$ be a compact K\"ahler manifold, $V\subset X$ an irreducible analytic $k$-dimensional subvariety, $T,S$ two closed $(1,1)$ currents on $X$ with $T\geq -A\omega, S\geq-A\omega$ which restrict to $V$. Then for any $p,q\geq 0, p+q\leq k,$ there is a constant $C$, which depends only on $X, V,\omega, p,q, A$ and on the cohomology classes of $T$ and $S$ such that
$$\left|\int_{V_{\rm reg}} (T|_{V_{\rm reg}})_{\rm ac}^p\wedge (S|_{V_{\rm reg}})_{\rm ac}^q \wedge \omega^{k-p-q}\right|\leq C.$$
\end{lem}
\begin{proof} Let $\mu:\ti{X}\to X$ be an embedded resolution of singularities of $V$, so $\ti{X}$ is compact K\"ahler manifold, $\mu$ is a composition of smooth blowups, the proper transform $\ti{V}$ of $V$ is smooth and $\mu|_{\ti{V}}:\ti{V}\to V$ is a modification. Given a K\"ahler metric $\ti{\omega}$ on $\ti{X}$, there is a constant $B>0$ such that $\mu^*\omega\leq B\ti{\omega}$. Then $\mu^*T$ is a well-defined closed $(1,1)$ current on $\ti{X}$ with $\mu^*T\geq -AB\ti{\omega}$.
Now $\mu$ is an isomorphism from a Zariski open subset of $\ti{V}$ to a Zariski open subset of $V$, and it identifies $(T|_{V_{\rm reg}})_{\rm ac}$ with $(\mu^*T|_{\ti{V}})_{\rm ac}$ on these open sets. But these are two $(1,1)$ forms with coefficients which are $L^1_{\rm loc}$ functions, and a Zariski closed set has Lebesgue measure zero. The same discussion applies to $(S|_{V_{\rm reg}})_{\rm ac}$ and $(\mu^*S|_{\ti{V}})_{\rm ac}$, and so
$$\int_{V_{\rm reg}} (T|_{V_{\rm reg}})_{\rm ac}^p\wedge (S|_{V_{\rm reg}})_{\rm ac}^q \wedge \omega^{k-p-q}=\int_{\ti{V}} (\mu^*T|_{\ti{V}})_{\rm ac}^p\wedge (\mu^*S|_{\ti{V}})_{\rm ac}^q \wedge \mu^*\omega^{k-p-q}.$$
To bound this integral, we apply Demailly's regularization procedure on $\ti{V}$ and obtain a sequence $T_j$ of smooth closed $(1,1)$ forms on $\ti{V}$, cohomologous to $\mu^*T|_{\ti{V}}$, such that
$$T_j\geq -AB\ti{\omega}-C\lambda_k\ti{\omega},$$
for a constant $C$ depending on $(\ti{V},\ti{\omega})$ and continuous functions $\lambda_k(x)$ which decrease to $\nu(\mu^*T|_{\ti{V}},x)$ for any $x\in \ti{V}$, and such that $T_j(x)\to (\mu^*T|_{\ti{V}})_{\rm ac}(x)$ a.e. as $j\to\infty$. By \cite[Lemma 2.5]{Bo} the Lelong numbers $\nu(\mu^*T|_{\ti{V}},x)$ are all bounded above by $C$, which depends only on $(\ti{V},\ti{\omega})$ and on the cohomology class of $\mu^*T|_{\ti{V}}$. Hence
$$T_j\geq -C_0\ti{\omega},$$
holds for all $j$.

Similarly, we construct smooth forms $S_j$ with
$$S_j\geq -C_0\ti{\omega}, \quad S_j(x)\to (\mu^*S|_{\ti{V}})_{\rm ac}(x) \textrm{ for a.e. }x\in \ti{V}.$$
But now note that the integrals
$$\int_{\ti{V}}(T_j+C_0\ti{\omega})^p\wedge (S_j+C_0\ti{\omega})^q \wedge \mu^*\omega^{k-p-q}$$
are nonnegative and bounded above by
$$B^{k-p-q}\int_{\ti{V}}(T_j+C_0\ti{\omega})^p\wedge (S_j+C_0\ti{\omega})^q \wedge \ti{\omega}^{k-p-q},$$
which is independent of $j$ since it is a cohomological number. Hence, by Fatou's lemma,
$$\int_{\ti{V}} (\mu^*T|_{\ti{V}}+C_0\ti{\omega})_{\rm ac}^p\wedge (\mu^*S|_{\ti{V}}+C_0\ti{\omega})_{\rm ac}^q \wedge \mu^*\omega^{k-p-q}\leq C,$$
and also
$$\int_{\ti{V}} (\mu^*T|_{\ti{V}}+C_0\ti{\omega})_{\rm ac}^p\wedge (\mu^*S|_{\ti{V}}+C_0\ti{\omega})_{\rm ac}^q \wedge \mu^*\omega^{k-p-q}\geq 0,$$
since the integrand is a nonnegative measure.
Now expand this as
$$\sum_{r,s}\binom{p}{r}\binom{q}{s}\int_{\ti{V}} (\mu^*T|_{\ti{V}})_{\rm ac}^r\wedge (C_0\ti{\omega})^{p-r}\wedge (\mu^*S|_{\ti{V}})_{\rm ac}^s\wedge(C_0\ti{\omega})^{q-s} \wedge \mu^*\omega^{k-p-q}.$$
We need to bound the term with $p=r,q=s$, and this follows from the estimate we just proved (which holds for all $p,q\geq 0, p+q\leq k$) together with a simple induction argument on $r+s$, the case of $r=s=0$ being trivial.
\end{proof}

\begin{rem}
We have chosen to use the absolutely continuous Monge-Amp\`ere operator to define the numerical restricted volume, but we could have equivalently used the non-pluripolar Monge-Amp\`ere operator $\langle \cdot\rangle$ as defined in \cite{BEGZ}, since these agree on currents with analytic singularities (indeed, both are equal to the extension by zero of the Monge-Amp\`ere operator on the complement of the singular set of the current). Also, as we will see in Remark \ref{nonpl}, if $V\not\subset E_{\rm nn}(\alpha)$ then we also have
\begin{equation}\label{nonp}
\langle \alpha^{k}\rangle_{X|V}=\lim_{\ve\downarrow 0}\int_{V_{\rm reg}} \langle (T_{min,\ve}|_{V_{\rm reg}})^k\rangle,
\end{equation}
where $T_{min,\ve}$ is a positive current with minimal singularities in the class $[\alpha+\ve\omega]$, for $\ve>0$, and where $\langle\cdot\rangle$ on the RHS is the non-pluripolar Monge-Amp\`ere operator.
\end{rem}

\begin{rem}
One can formally also define the numerical restricted volume when $\dim V=0$, i.e. $V$ is a point, by setting $\langle\alpha^0\rangle_{X|V}$ equal to $0$ if $V\subset E_{\rm nn}(\alpha)$ and equal to $1$ otherwise, and all theorems that we will prove are still valid in this case. However, we will not insist on this, and from now on $V$ will always denote a positive-dimensional irreducible subvariety.
\end{rem}

\begin{rem}
As in \cite{BoT}, one can similarly also define a restricted numerical intersection product
$$\langle \alpha_1\cdots\alpha_k\rangle_{X|V},$$
which extends the algebraic construction in \cite{BFJ}, where the $[\alpha_j]$'s are pseudoeffective classes with $V\not\subset \cup_j E_{\rm nn}(\alpha_j).$
\end{rem}

As mentioned in the introduction, the numerical restricted volume coincides with the ``mobile intersection number''
$$\langle\alpha^k\rangle_{X|V}=(\alpha^k\cdot [V])_{\geq 0},$$
defined in Boucksom's thesis \cite[Definition 3.2.1]{BoT}, where he proves some basic properties. For the reader's convenience, we incorporate some of these here.
The first property \cite[Proposition 3.2.4]{BoT} is a semicontinuity result:
\begin{lem}\label{semi}
Let $(X,\omega)$ be a compact K\"ahler manifolds and $[\alpha_j]$ be a sequence of pseudoeffective classes which converge to a class $[\alpha]$. If $V$ is an irreducible $k$-dimensional subvariety of $X$, with $V\not\subset E_{\rm nn}(\alpha_j)$ for all $j$, then we have
\begin{equation}\label{semic}
\langle\alpha^k\rangle_{X|V}\geq \limsup_{j\to\infty}\langle\alpha_j^k\rangle_{X|V}.
\end{equation}
\end{lem}
\begin{proof}
First note that thanks to \cite[Proposition 3.5]{Bo2}, the minimal multiplicities $\nu([\alpha_j],x)$ are lower semicontinuous
$$\nu([\alpha],x)\leq\liminf_{j\to\infty}\nu([\alpha_j],x),$$
for all $x\in X$. Since $V\not\subset E_{\rm nn}(\alpha_j)$ for all $j$, and each such set is an at worst countable union of Zariski closed sets (but not finite in general \cite{Les}), we conclude that there is a point
$x\in V$ with $\nu([\alpha_j],x)=0$ for all $j$. Therefore $\nu([\alpha],x)=0$ too, which shows that $V\not\subset E_{\rm nn}(\alpha)$. In particular, the restricted volume $\langle\alpha^k\rangle_{X|V}$ is defined by \eqref{voldef}.

For each $j$ we may choose $T_j\in [\alpha_j]$ with $T_j\geq -\ve_j\omega$, $\ve_j\to 0$, with analytic singularities not containing $V$, such that
$$\left|\int_{V_{\rm reg}} (T_j+\ve_j\omega)_{\rm ac}^k-\langle\alpha_j^k\rangle_{X|V}\right|\to 0,$$
as $j\to\infty$.  Next, choose a sequence of closed smooth $(1,1)$ forms $\theta_j$ in the class $[\alpha-\alpha_j]$ which converge smoothly to zero.
In particular, $\theta_j\geq -\delta_j\omega$ for some $\delta_j\to 0$. Then $T_j+\theta_j$ is a closed $(1,1)$ current in the class $[\alpha]$, with analytic singularities not containing $V$, and with
$T_j+\theta_j\geq -(\ve_j+\delta_j)\omega$, and so
$$\limsup_{j\to\infty} \int_{V_{\rm reg}} (T_j+\theta_j+(\ve_j+\delta_j)\omega)_{\rm ac}^k\leq\langle\alpha^k\rangle_{X|V},$$
from the definition. But the difference
$$\left|\int_{V_{\rm reg}} (T_j+\ve_j\omega)_{\rm ac}^k- \int_{V_{\rm reg}} (T_j+\theta_j+(\ve_j+\delta_j)\omega)_{\rm ac}^k\right|$$
is easily seen to go to zero, using Lemma \ref{finite}. This finishes the proof.
\end{proof}
Although we won't use this, continuity in fact holds in \eqref{semic} if $V\not\subset E_{\rm nK}(\alpha)$, by \cite[Proposition 3.2.4]{BoT}.

The second property is taken from \cite[Proposition 3.2.2]{BoT} (see also \cite[Proposition 4.5]{Ma} for a weaker statement).

\begin{lem}\label{l4}
If $[\alpha]$ is nef and $V$ is an irreducible $k$-dimensional subvariety of $X$ then we have
$$\langle\alpha^k\rangle_{X|V}=\int_V\alpha^k.$$
\end{lem}
\begin{proof}
Since $[\alpha]$ is nef, for any $\ve>0$ there is a closed smooth form $\rho_\ve$ in $[\alpha]$ with $\rho_\ve\geq-\ve\omega$. In particular,
$$\sup_{T\in[\alpha],T\geq -\ve\omega, V\not\subset E_+(T)} \int_{V_{\rm reg}}\left(T|_{V_{\rm reg}}+\ve\omega\right)_{\rm ac}^k\geq \int_{V_{\rm reg}}(\rho_\ve+\ve\omega)^k=\int_V(\alpha+\ve\omega)^k,$$
and letting $\ve\to 0$ we get $\langle\alpha^k\rangle_{X|V}\geq\int_V\alpha^k$.
We want to show the reverse inequality. Fix $\ve>0$ and pick $T_\ve\in[\alpha]$ with analytic singularities, with $T_\ve\geq -\ve\omega$ and $V\not\subset E_+(T_\ve)$, such that
$$\int_{V_{\rm reg}}\left(T_\ve|_{V_{\rm reg}}+\ve\omega\right)_{\rm ac}^k\geq \sup_{T\in[\alpha],T\geq -\ve\omega, V\not\subset E_+(T)} \int_{V_{\rm reg}}\left(T|_{V_{\rm reg}}+\ve\omega\right)_{\rm ac}^k-\ve.$$
As in the proof of Lemma \ref{finite}, up to considering an embedded resolution of singularities of $V$, we may assume that $V$ is smooth.
Then $(\rho_\ve+\ve\omega)|_V$ is less singular than $(T_\ve+\ve\omega)|_V$, and we can apply \cite[Theorem 1.16]{BEGZ} to see that
$$\int_V (\rho_\ve+\ve\omega)|_V^k\geq \int_V \langle (T_\ve+\ve\omega)|_V^k\rangle,$$
where $\langle\cdot\rangle$ is the non-pluripolar product. But if $T$ is a positive current with analytic singularities, we have $\langle T^k\rangle=T_{\rm ac}^k$, and so
\[\begin{split}
\int_V(\alpha+\ve\omega)^k&=\int_V (\rho_\ve+\ve\omega)^k\geq \int_V ((T_\ve+\ve\omega)|_V)^k_{\rm ac}\\
&\geq \sup_{T\in[\alpha],T\geq -\ve\omega, V\not\subset E_+(T)} \int_{V}\left(T|_{V}+\ve\omega\right)_{\rm ac}^k-\ve.
\end{split}\]
Letting $\ve\to 0$ we conclude that $\langle\alpha^k\rangle_{X|V}\leq\int_V\alpha^k$, as required.
\end{proof}

The third property, following \cite[Lemma 3.2.5]{BoT}, shows that when $V\not\subset E_{\rm nK}(\alpha)$ then the definition of numerical restricted volume simplifies:
\begin{lem}\label{volr}
For any pseudoeffective class $[\alpha]$ and any irreducible $k$-dimensional subvariety $V$ which satisfies $V\not\subset E_{\rm nK}(\alpha)$ (so in particular $[\alpha]$ is big), we have
\begin{equation}\label{toshow}
\langle\alpha^k\rangle_{X|V}=\sup_{T} \int_{V_{\rm reg}}\left(T|_{V_{\rm reg}}\right)_{\rm ac}^k,
\end{equation}
where the supremum is over all K\"ahler currents (or equivalently all closed positive currents) in the class $[\alpha]$ with analytic singularities which do not contain $V$.
\end{lem}
\begin{proof}
Clearly
$$\langle\alpha^k\rangle_{X|V}\geq \sup_{T\in[\alpha],T\geq 0, V\not\subset E_+(T)} \int_{V_{\rm reg}}\left(T|_{V_{\rm reg}}\right)_{\rm ac}^k=:A,$$
where the supremum is over all closed positive currents in the class $[\alpha]$ with analytic singularities which do not contain $V$.
If we do not have equality, then there would exist $\delta>0$ such that
$$\sup_{T\in[\alpha],T\geq -\ve\omega, V\not\subset E_+(T)} \int_{V_{\rm reg}}\left(T|_{V_{\rm reg}}+\ve\omega\right)_{\rm ac}^k\geq A+2\delta,$$
for $\ve>0$ arbitrarily small. For any such $\ve$ choose $T_\ve\in [\alpha]$ with analytic singularities, with $T_\ve\geq -\ve\omega$, $V\not\subset E_+(T_\ve)$ and
$$ \int_{V_{\rm reg}}\left(T_{\ve}|_{V_{\rm reg}}+\ve\omega\right)_{\rm ac}^k\geq A+\delta.$$
Let $S\in[\alpha]$ be a K\"ahler current with analytic singularities with $S\geq\eta\omega,\eta>0$, and with $V\not\subset E_+(S)$, which exists thanks to the assumption that
$V\not\subset E_{\rm nK}(\alpha)$. Choose $\gamma>0$ small enough so that
$$(1-\gamma)^k(A+\delta)\geq A+\frac{\delta}{2}.$$
Applying Lemma \ref{finite} we see that
$$\left|\int_{V_{\rm reg}} (T_\ve|_{V_{\rm reg}})_{\rm ac}^i\wedge (S|_{V_{\rm reg}})_{\rm ac}^{k-i}\right|\leq C,$$
for some fixed constant independent of $\ve$ (small). We can therefore choose $\gamma$ sufficiently small so that we also have
$$\sum_{i=0}^{k-1}\gamma^{k-i}(1-\gamma)^i\binom{k}{i}\left|\int_{V_{\rm reg}} (T_\ve|_{V_{\rm reg}})_{\rm ac}^i\wedge (S|_{V_{\rm reg}})_{\rm ac}^{k-i}\right|<\frac{\delta}{8},$$
for all $0<\ve<1.$
Fixing any $\ve<\gamma\eta/2$, we have that
$$(1-\gamma)T_\ve+\gamma S\geq \left(-\ve(1-\gamma)+\gamma\eta\right)\omega\geq \left(-\ve+\gamma\eta\right)\omega\geq \frac{\gamma\eta}{2}\omega,$$
so that $(1-\gamma)T_\ve+\gamma S\in[\alpha]$ is a K\"ahler current with analytic singularities which do not contain $V$ and
\[\begin{split}
A&\geq\int_{V_{\rm reg}}\left(((1-\gamma)T_\ve+\gamma S)|_{V_{\rm reg}}\right)_{\rm ac}^k\\
&\geq (1-\gamma)^k\int_{V_{\rm reg}}\left(T_\ve|_{V_{\rm reg}}\right)_{\rm ac}^k-\frac{\delta}{8}\\
&\geq (1-\gamma)^k\int_{V_{\rm reg}}\left(T_\ve|_{V_{\rm reg}}+\ve\omega\right)_{\rm ac}^k-\frac{\delta}{4}\\
&\geq (1-\gamma)^k(A+\delta)-\frac{\delta}{4}\geq  A+\frac{\delta}{4},
\end{split}\]
a contradiction, where we chose $\ve$ small enough so that
$$(1-\gamma)^k\int_{V_{\rm reg}}\left(T_\ve|_{V_{\rm reg}}\right)_{\rm ac}^k\geq (1-\gamma)^k\int_{V_{\rm reg}}\left(T_\ve|_{V_{\rm reg}}+\ve\omega\right)_{\rm ac}^k-\frac{\delta}{8},$$
which is possible again thanks to Lemma \ref{finite}. Since $(1-\gamma)T_\ve+\gamma S\in[\alpha]$ is a K\"ahler current, this argument also proves \eqref{toshow} when we restrict to K\"ahler currents.
\end{proof}
\begin{rem}
When $[\alpha]$ is big and $V\not\subset E_{\rm nK}(\alpha)$, the numerical restricted volume was introduced independently by Hisamoto \cite{Hi} and Matsumura \cite{Ma}, by the right hand side of \eqref{toshow} (using closed positive currents). They also proved some of the basic properties of numerical restricted volumes in this case. For example, the arguments in \cite[Theorem 1.3]{Hi} can be used to show that when $[\alpha]$ is $V$-big we have
$$\langle\alpha^k\rangle_{X|V}=\int_{V_{\rm reg}}\langle (T_{min}|_{V_{\rm reg}})^k\rangle,$$
where $\langle\cdot\rangle$ on the RHS is the non-pluripolar Monge-Amp\`ere operator. Our main interest is however in the case when $V$ is possibly contained in $E_{\rm nK}(\alpha)$, so $[\alpha]$ is $V$-psef but not $V$-big.
\end{rem}
The following result was proved independently by Hisamoto \cite{Hi} and Matsumura \cite{Ma}, using the crucial \cite[Theorem 2.13]{ELMNP}:
\begin{thm}\label{hisamat}
Let $X$ be a projective manifold, $L$ a big line bundle and $V\subset X$ an irreducible subvariety. If $V\not\subset E_{\rm nK}(c_1(L))=\mathbb{B}_+(L)$, then we have
\begin{equation}\label{equal3}
\langle c_1(L)^{\dim V}\rangle_{X|V}=\vol_{X|V}(L).
\end{equation}
\end{thm}
By homogeneity, this result also holds when $L$ is just a $\mathbb{Q}$-divisor.
As remarked in the introduction, \eqref{equal3} remains also trivially true if $V\subset E_{\rm nn}(c_1(L))=\mathbb{B}_-(L)$ since both sides are zero. In general however \eqref{equal3} fails, see Examples \ref{exx} and \ref{exx2} below.

The correct substitute for \eqref{equal3} is given by Theorem \ref{regg} which we now prove.
\begin{proof}[Proof of Theorem \ref{regg}]
If $V\subset \mathbb{B}_{-}(L)=E_{\rm nn}(c_1(L))$ then $\langle c_1(L)^{\dim V}\rangle_{X|V}=0$ by definition, while also by definition we have $V\subset \mathbb{B}(L+\ve H)$ for all $\ve\in\mathbb{Q}_{>0}$, and so
$\vol_{X|V}(L+\ve H)=0$ too.

If instead $V\not\subset E_{\rm nn}(c_1(L))$, then essentially by definition
$$\langle c_1(L)^{\dim V}\rangle_{X|V}=\lim_{\ve\to 0,\ \ve\in\mathbb{Q}_{>0}}\sup_{T\in c_1(L+\ve H), T\geq 0, V\not\subset E_+(T)}\int_{V_{\rm reg}}(T|_{V_{\rm reg}})_{\rm ac}^{\dim V},$$
where the currents $T$ have analytic singularities. But we also have that $V\not\subset \mathbb{B}_+(L+\ve H)$ for all $\ve\in\mathbb{Q}_{>0},$ and so Lemma \ref{volr} and Theorem \ref{hisamat} give
$$\sup_{T\in c_1(L+\ve H), T\geq 0, V\not\subset E_+(T)}\int_{V_{\rm reg}}(T|_{V_{\rm reg}})_{\rm ac}^{\dim V}=\langle c_1(L+\ve H)^{\dim V}\rangle_{X|V}=\vol_{X|V}(L+\ve H),$$
and the result follows. Finally, inequality \eqref{oneway} follows from this and the fact that $\vol_{X|V}(L)\leq \vol_{X|V}(L+\ve H)$.
\end{proof}
\begin{rem}
The quantity on the RHS of \eqref{regular} has also recently been studied in \cite{CHPW}, where it is denoted by $\vol^+_{X|V}(L)$ (see their Definition 2.2).
\end{rem}

\begin{ex}\label{exx}
We give an example of a projective manifold $X$ with a big line bundle $L$ and an irreducible subvariety $V\subset \mathbb{B}_+(L)$ such that
\begin{equation}\label{different}
\langle c_1(L)^{\dim V}\rangle_{X|V}\neq\vol_{X|V}(L).
\end{equation}
which also shows that $\vol_{X|V}(L)$ is not in general a numerical invariant of the line bundle $L$ (more examples can be found in \cite[Example 5.10]{ELMNP}, see Example \ref{exx2} below, and in \cite[Example 2.3]{CHPW}).

Let $S$ be the ruled surface described in \cite[Example 10.3.3]{Laz}, which possesses two numerically equivalent nef and big line bundles $L_1,L_2$ with
$\mathbb{B}(L_1)=\emptyset, \mathbb{B}(L_2)\neq\emptyset.$ Let $X=S\times\mathbb{CP}^1$, with projections $\pi_1,\pi_2$, and let
$$\ti{L}_1=\pi_1^*L_1+\pi_2^*\mathcal{O}(1), \quad \ti{L}_2=\pi_1^*L_2+\pi_2^*\mathcal{O}(1).$$
These line bundles are still nef and big and numerically equivalent. We have that $\mathbb{B}(\ti{L}_1)=\emptyset$, while if $p\in\mathbb{B}(L_2)$ then
$V:=\{p\}\times\mathbb{CP}^1\subset \mathbb{B}(\ti{L}_2)$.
Since $\ti{L}_1,\ti{L}_2$ are nef and numerically equivalent, we obtain from Lemma \ref{l4} that
$$\langle c_1(\ti{L}_1)\rangle_{X|V}=\langle c_1(\ti{L}_2)\rangle_{X|V}=\int_{V}c_1(\ti{L}_1)=1,$$
while clearly
$$\vol_{X|V}(\ti{L}_2)=0,$$
since $V\subset \mathbb{B}(\ti{L}_2)$, and so $L:=\ti{L}_2$ gives an example of \eqref{different} Lastly, we check that
$$\vol_{X|V}(\ti{L}_1)>0,$$
which shows that the restricted volume is not a numerical invariant.
Indeed, since $\mathbb{B}(L_1)=\emptyset$, we can find a section $s\in H^0(S,kL_1), k\geq 1,$ with $s(p)\neq 0$.
Note that $k\ti{L}_1|_{V}\cong\mathcal{O}(k)$, so given any $s'\in H^0(V,k\ti{L}_1|_{V})=H^0(\mathbb{CP}^1,\mathcal{O}(k))$ we can regard $\pi_1^*s\otimes \pi_2^*s'$ as a section in
$H^0(X,k\ti{L}_1)$ which restricts to $s'$ on $V$. Therefore every section of $k\ti{L}_1|_{V}$ lifts to $X$, and the same holds for sections of $\ell k\ti{L}_1$, for every $\ell\geq 1$, and so we conclude that $\vol_{X|V}(\ti{L}_1)>0.$
\end{ex}

\section{The main conjecture}\label{sect3}
The following is our motivating problem, and is equivalent to Conjecture \ref{conjec}:
\begin{conj}\label{con3}
Let $X$ be a compact K\"ahler manifold, $[\alpha]$ a pseudoeffective class on $X$ and $V$ be an irreducible component of $E_{\rm nK}(\alpha)$ (which is necessarily positive-dimensional).
Then
$$\langle \alpha^{\dim V}\rangle_{X|V}=0.$$
\end{conj}
Note that this is obviously true if $[\alpha]$ is not big.
In general, Conjecture \ref{con3} is equivalent to the equality
$$E_{\rm nK}(\alpha)=\Null(\alpha),$$
in Conjecture \ref{conjec} because it is immediate to see that if $V\not\subset E_{\rm nK}(\alpha)$ then $\langle\alpha^{\dim V}\rangle_{X|V}>0$.
Using Lemma \ref{l4}, we see that Conjecture \ref{con3} holds when $[\alpha]$ is nef, thanks to the main theorem of our earlier work \cite{CT}.

As a simple corollary of this conjecture, we obtain the transcendental analog of \cite[Theorems 5.2 (b) and 5.7]{ELMNP}:
\begin{prop}\label{con} Let $X$ be a compact K\"ahler manifold, $[\alpha]$ a pseudoeffective class on $X$ and $V$ be an irreducible component of $E_{\rm nK}(\alpha)$.
Assume Conjecture \ref{con3}. Then for any sequence of pseudoeffective classes $\alpha_j$ that converge to $\alpha$, we have that
$$\langle \alpha_j^{\dim V}\rangle_{X|V}\to 0.$$
\end{prop}
\begin{proof}
Suppose this is false, so that there exists a sequence of pseudoeffective classes $\alpha_j$ which converge to $\alpha$ and with
$$\langle \alpha_j^{\dim V}\rangle_{X|V}\geq \ve>0,$$
for all $j$. In particular we have $V\not\subset E_{\rm nn}(\alpha_j)$.
Applying Lemma \ref{semi} we conclude that $\langle \alpha^{\dim V}\rangle_{X|V}\geq \ve>0$, a contradiction to Conjecture \ref{con3}.
\end{proof}

Conjecture \ref{con3} would also give a new proof of a theorem of Ein-Lazarsfeld-Musta\c{t}\u{a}-Nakamaye-Popa \cite[Theorem 5.7]{ELMNP}, the main theorem of their paper:
\begin{prop}\label{prob2}
Let $X$ be a projective manifold and $D$ a big $\mathbb{R}$-divisor on $X$. Assume Conjecture \ref{con3}. If $V$ is one of the irreducible components of $E_{\rm nK}(D)$, then
\begin{equation}\label{wan}
\lim_{D'\to D}\vol_{X|V}(D')=0,
\end{equation}
where the limit is over all $\mathbb{Q}$-divisors $D'$ whose classes converge to the class of $D$.
\end{prop}
\begin{proof}
Let $D'\to D$ be $\mathbb{Q}$-divisors as in the statement.
Thanks to \eqref{oneway}, which by homogeneity holds for $\mathbb{Q}$-divisors, we have that
$$\vol_{X|V}(D')\leq \langle c_1(D')^{\dim V}\rangle_{X|V},$$
and so the result follows from Proposition \ref{con}.
\end{proof}
Conversely, \cite[Theorem 5.7]{ELMNP} implies Conjecture \ref{con3} when $X$ is projective and $[\alpha]=c_1(D)$ for $D$ a big $\mathbb{R}$-divisor:
\begin{proof}[Proof of Proposition \ref{prob1}]
Let $V$ be an irreducible component of $E_{\rm nK}(c_1(D))$, of dimension $k>0$, and fix an ample $\mathbb{Q}$-divisor $H$ such that $D+\lambda H$ is a $\mathbb{Q}$-divisor for some $\lambda\in\mathbb{R}_{>0}$. We may assume that $V\not\subset E_{\rm nn}(c_1(D))$, otherwise the result is trivial. Then from the definition we have
$$\langle c_1(D)^k\rangle_{X|V}=\lim_{\ve\downarrow 0}\sup_{T\in c_1(D+\ve H), T\geq 0,V\not\subset E_+(T)}\int_{V_{\rm reg}}(T|_{V_{\rm reg}})_{\rm ac}^k,$$
where the currents $T$ have analytic singularities. By construction, there is a sequence $\ve_j\in\mathbb{R}_{>0}$ with $\ve_j\to 0$ such that $D+\ve_jH$ is a $\mathbb{Q}$-divisor, and so suppressing the index $j$ from the notation we may assume that $D+\ve H$ is a $\mathbb{Q}$-divisor.

Then note that $V\not\subset E_{\rm nK}(D+\ve H)$ for all such $\ve>0$, and so Lemma \ref{volr} and Theorem \ref{hisamat} give
$$\sup_{T\in c_1(D+\ve H), T\geq 0,V\not\subset E_+(T)}\int_{V_{\rm reg}}(T|_{V_{\rm reg}})_{\rm ac}^k=\langle c_1(D+\ve H)^{\dim V}\rangle_{X|V}=\vol_{X|V}(D+\ve H).$$
But now \cite[Theorem 5.7]{ELMNP} says that $\lim_{\ve\to 0}\vol_{X|V}(D+\ve H)=0$, and the result follows.
\end{proof}

We close this section with the very simple proof of Theorem \ref{localamp}:
\begin{proof}[Proof of Theorem \ref{localamp}]
Given any $x\in U$ let $V$ be any irreducible subvariety of $X$ which passes through $x$, and say $\dim V=k>0$. By definition we have
$$\langle\alpha^k\rangle_{X|V}\geq \int_{V_{\rm reg}} (T|_V)_{\rm ac}^k>0,$$
where the last inequality follows from the fact that $T$ is a smooth K\"ahler metric near $x$. Since $V$ is arbitrary, it follows from Conjecture \ref{con3} that $x\not\in E_{\rm nK}(\alpha)$, as required.
\end{proof}

\section{Technical lemmas}\label{sect4}
In this section we prove some technical results, which will be used in the proofs of the main theorems. First of all, we note the following (cf. \cite[Corollary 4.6]{CL} in the algebraic case):
\begin{lem}\label{nonnef}
If $\mu:\ti{X}\to X$ is a surjective holomorphic map between compact K\"ahler manifolds and $[\alpha]$ is a big class on $X$, then
$$E_{\rm nn}(\mu^*\alpha)=\mu^{-1}(E_{\rm nn}(\alpha)).$$
\end{lem}
\begin{proof}
Let $T$ be a closed positive current in the class $[\alpha]$ with minimal singularities, so by definition $E_{\rm nn}(\alpha)=\{x\in X\ |\ \nu(T,x)>0\}$.
Thanks to \cite[Proposition 1.12]{BEGZ}, $\mu^*T$ has minimal singularities in the class $[\mu^*\alpha]$, so
$E_{\rm nn}(\mu^*\alpha)=\{x\in \ti{X}\ |\ \nu(\mu^*T,x)>0\}$. A result of Favre \cite{Fa} and Kiselman \cite{Ki} implies that $\nu(\mu^*T,x)>0$ iff $\nu(T,\mu(x))>0$, and the result follows.
\end{proof}

The next result that we will need is the bimeromorphic invariance of the numerical restricted volume, which is analogous to \cite[Lemma 2.4]{ELMNP} in the algebraic case:
\begin{lem}\label{invar}
Let $\mu:\ti{X}\to X$ be a modification between compact K\"ahler manifolds, which is a composition of blowups with smooth centers, and let $V\subset X$ be a $k$-dimensional irreducible subvariety not contained
in $\mu(\mathrm{Exc}(\mu))$, so that its proper transform $\ti{V}$ is well-defined and $\mu|_{\ti{V}}:\ti{V}\to V$ is a modification.
If $[\alpha]$ is a big $(1,1)$ class on $X$ then
$$\langle\alpha^k\rangle_{X|V}=\langle\mu^*\alpha^k\rangle_{\ti{X}|\ti{V}}.$$
\end{lem}
\begin{proof}
If $V\subset E_{\rm nn}(\alpha)$ then by definition $\langle\alpha^k\rangle_{X|V}=0$. Since $E_{\rm nn}(\mu^*\alpha)=\mu^{-1}(E_{\rm nn}(\alpha))$ by Lemma \ref{nonnef}, we see that
$\ti{V}\subset E_{\rm nn}(\mu^*\alpha)$ and so $\langle\mu^*\alpha^k\rangle_{\ti{X}|\ti{V}}=0$ too. So we assume that $V\not\subset E_{\rm nn}(\alpha)$, so that
$$\langle\alpha^k\rangle_{X|V}=\lim_{\ve\downarrow 0} \sup_{T\in[\alpha],T\geq-\ve\omega,V\not\subset E_+(T)}\int_{V_{\rm reg}}(T|_{V_{\rm reg}}+\ve\omega)_{\rm ac}^k.$$
If $T$ is any such current, then $\mu^*T\in [\mu^*\alpha]$ has analytic singularities which don't contain $\ti{V}$, and satisfies
$\mu^*T\geq -\ve\mu^*\omega$. Fix $\ti{\omega}$ a K\"ahler metric on $\ti{X}$. Then $\mu^*\omega\leq C\ti{\omega}$ for some constant $C$, and so
$\mu^*T\geq -\ve C\ti{\omega}$. Since $\mu|_{\ti{V}}:\ti{V}\to V$ is a modification we have
\[\begin{split}
\int_{V_{\rm reg}}(T|_{V_{\rm reg}}+\ve\omega)_{\rm ac}^k&=\int_{\ti{V}_{\rm reg}}(\mu^*T|_{\ti{V}_{\rm reg}}+\ve\mu^*\omega)_{\rm ac}^k
\leq\int_{\ti{V}_{\rm reg}}(\mu^*T|_{\ti{V}_{\rm reg}}+\ve C\ti{\omega})_{\rm ac}^k,
\end{split}\]
and taking the supremum over all $T$ and letting $\ve\to 0$ we conclude that
$$\langle\alpha^k\rangle_{X|V}\leq\langle\mu^*\alpha^k\rangle_{\ti{X}|\ti{V}}.$$
For the converse, let $E=\mathrm{Exc}(\mu)$ be the union of all the exceptional divisors of $\mu$. Then there are a smooth form $\eta$,
a quasi-psh function $\psi$ and $\delta>0$ small such that
$$\eta=[E]-\ddbar\psi,$$
and
$$\ti{\omega}:=\mu^*\omega-\delta\eta,$$
is a K\"ahler metric on $\ti{X}$, see e.g. \cite[Lemma 6]{PS}.
Let $\ti{T}$ be a current on $\ti{X}$ in $[\mu^*\alpha]$, with analytic singularities which don't contain $\ti{V}$ and with
$\ti{T}\geq -\ve\ti{\omega}$. Then $\ti{T}-\ve\delta\eta\geq -\ve\mu^*\omega,$ and since the current $\mu_*\eta$ is cohomologous to zero on $X$, we see that
$T:=\mu_*(\ti{T}-\ve\delta\eta)\geq -\ve\omega$ is a current on $X$ in the class $[\alpha]$, which is smooth at the generic point of $V$.
We have
\[\begin{split}
\int_{\ti{V}_{\rm reg}}(\ti{T}|_{\ti{V}_{\rm reg}}+\ve\ti{\omega})_{\rm ac}^k&= \int_{\ti{V}_{\rm reg}}(\ti{T}|_{\ti{V}_{\rm reg}}-\ve\delta\eta+\ve\mu^*\omega)_{\rm ac}^k=\int_{V_{\rm reg}}(T|_{V_{\rm reg}}+\ve\omega)_{\rm ac}^k.
\end{split}\]
If $T_j$ is a Demailly regularization of $T$, then $T_j\geq -2\ve\omega$ for all $j$ large, $T_j\in[\alpha]$ have analytic singularities which do not contain $V$, and
$(T_j)_{\rm ac}(x)\to T_{\rm ac}(x)$ pointwise for all $x\in V$ generic. By Fatou's lemma,
$$\liminf_{j\to\infty}\int_{V_{\rm reg}}(T_j|_{V_{\rm reg}}+2\ve\omega)_{\rm ac}^k\geq \int_{V_{\rm reg}}(T|_{V_{\rm reg}}+2\ve\omega)_{\rm ac}^k\geq \int_{V_{\rm reg}}(T|_{V_{\rm reg}}+\ve\omega)_{\rm ac}^k.$$
Choosing $j$ large, we get
$$\int_{\ti{V}_{\rm reg}}(\ti{T}|_{\ti{V}_{\rm reg}}+\ve\ti{\omega})_{\rm ac}^k\leq \int_{V_{\rm reg}}(T_j|_{V_{\rm reg}}+2\ve\omega)_{\rm ac}^k+\ve.$$
Taking the supremum over all $\ti{T}$ and letting $\ve\to 0$ we conclude that
$$\langle\mu^*\alpha^k\rangle_{\ti{X}|\ti{V}}\leq\langle\alpha^k\rangle_{X|V}.$$
\end{proof}
We now recall the following lemma proved in \cite[Proposition 2.5]{To3}:
\begin{lem}\label{pull}
Let $\mu:\ti{X}\to X$ be a modification between compact K\"ahler manifolds. If $[\alpha]$ is any $(1,1)$ class on $X$ then
$$E_{\rm nK}(\mu^*\alpha)=\mu^{-1}(E_{\rm nK}(\alpha))\cup\mathrm{Exc}(\mu).$$
\end{lem}
The following result is the analog for the null locus:
\begin{lem}\label{pull2}
Let $\mu:\ti{X}\to X$ be a modification between compact K\"ahler manifolds, which is a composition of blowups with smooth centers. If $[\alpha]$ is any $(1,1)$ class on $X$ then
$$\Null(\mu^*\alpha)=\mu^{-1}(\Null(\alpha))\cup\mathrm{Exc}(\mu).$$
\end{lem}
\begin{proof}
Since $[\alpha]$ is big if and only if $[\mu^*\alpha]$ is big \cite[Proposition 4.12]{Bo}, the statement is only nontrivial when $[\alpha]$ is big, which we assume.
First we show that $\Null(\mu^*\alpha)\subset\mu^{-1}(\Null(\alpha))\cup\mathrm{Exc}(\mu)$. Let $x\not\in \mu^{-1}(\Null(\alpha))\cup\mathrm{Exc}(\mu)$ and let $\ti{V}$ be any irreducible subvariety of $\ti{X}$ through $x$. Then $V:=\mu(\ti{V})$ is an irreducible subvariety of $X$, whose proper transform equals $\ti{V}$, and $\langle\alpha^{\dim V}\rangle_{X|V}>0$. Lemma \ref{invar} gives $\langle(\mu^*\alpha)^{\dim V}\rangle_{\ti{X}|\ti{V}}>0$. Since $\ti{V}$ was arbitrary, we conclude that $x\not\in \Null(\mu^*\alpha)$.

Conversely, we show that $\mu^{-1}(\Null(\alpha))\cup\mathrm{Exc}(\mu)\subset\Null(\mu^*\alpha)$. First, let us see that $\mathrm{Exc}(\mu)\subset\Null(\mu^*\alpha)$. If $x\in \mathrm{Exc}(\mu),$ Zariski's main theorem
implies that the fiber $\mu^{-1}(\mu(x))$ is connected and positive dimensional. Let $E$ be an irreducible component of $\mu^{-1}(\mu(x))$, and let $\mu':X'\to\ti{X}$ be an embedded resolution of singularities of $E$, so that its proper transform
$E'\subset X'$ is smooth and connected and $\mu'|_{E'}:E'\to E$ is bimeromorphic. Then Lemma \ref{invar} implies that $\langle(\mu'^*\mu^*\alpha)^{\dim E}\rangle_{X'|E'}=\langle(\mu^*\alpha)^{\dim E}\rangle_{\ti{X}|E}$. We also have the trivial inequality
$$\langle(\mu'^*\mu^*\alpha)^{\dim E}\rangle_{X'|E'}\leq \vol_{E'}(\mu'^*\mu^*\alpha|_{E'}),$$
but the class $\mu'^*\mu^*\alpha|_{E'}$ is zero since $E'$ is contained in a fiber of $\mu\circ\mu',$ and so
$$\langle(\mu'^*\mu^*\alpha)^{\dim E}\rangle_{X'|E'}=0=\langle(\mu^*\alpha)^{\dim E}\rangle_{\ti{X}|E},$$ and we conclude that $x\in \Null(\mu^*\alpha)$.

Let now $x\in \mu^{-1}(\Null(\alpha))\backslash\mathrm{Exc}(\mu)$. By definition $\mu$ is an isomorphism near $x$, and there is an irreducible subvariety $V$ of $X$ through $\mu(x)$ with $\langle\alpha^{\dim V}\rangle_{X|V}=0$. If $\ti{V}$ denotes the proper transform of $V$, which passes through $x$, then Lemma \ref{invar} gives $\langle(\mu^*\alpha)^{\dim V}\rangle_{\ti{X}|\ti{V}}=\langle\alpha^{\dim V}\rangle_{X|V}=0$, and so $x\in \Null(\mu^*\alpha)$.
\end{proof}

\section{Fujita approximation}\label{sect5}
To establish a Fujita-type approximation result for the numerical restricted volume, the following lemma is the key:
\begin{lem}\label{simple}
Let $(X,\omega)$ be a compact K\"ahler manifold and $[\alpha]$ a pseudoeffective class. Then for every $\ve>0$ small there exists
a closed $(1,1)$ current $T_\ve\in[\alpha]$, with $T_\ve\geq -\ve\omega$, with analytic singularities contained in $E_{\rm nn}(\alpha)$, and such that
for all irreducible subvarieties with $V\not\subset E_{\rm nn}(\alpha)$ we have
\begin{equation}\label{simp1}
\int_{V_{\rm reg}} (T_\ve|_{V_{\rm reg}}+\ve\omega)_{\rm ac}^{\dim V}\geq \langle\alpha^{\dim V}\rangle_{X|V},
\end{equation}
and
\begin{equation}\label{simp2}
\lim_{\ve\downarrow 0}\int_{V_{\rm reg}} (T_\ve|_{V_{\rm reg}}+\ve\omega)_{\rm ac}^{\dim V}=\langle\alpha^{\dim V}\rangle_{X|V}.
\end{equation}
\end{lem}
\begin{proof}
Let $V$ be any such subvariety, with $k=\dim V>0$.
Recalling \eqref{moron}, we see that $V\not\subset E_{\rm nK}(\alpha+\ve\omega),$ for all $\ve>0$, and so if $T_{min,\ve}$ is any positive current with minimal singularities in the class $[\alpha+\ve\omega]$ then $T_{min,\ve}$ has locally bounded potentials on a Zariski open subset of $X$ which contains the complement of $E_{\rm nn}(\alpha)$.
Applying Demailly's regularization \cite{Dem92} to $T_{min,\frac{\ve}{2}}-\frac{\ve}{2}\omega$ we obtain currents $T_\ve\geq -\ve\omega$ in the class $[\alpha]$ with analytic singularities contained in $E_{\rm nn}(\alpha)$, and so that $T_\ve$ is less singular than $T_{min,\frac{\ve}{2}}$ (since the potentials along Demailly's regularization procedure decrease to the original current), and $T_\ve$ can thus be restricted to $V$.

If $S_\ve$ is any current in $[\alpha]$ with analytic singularities not containing $V$ and with $S_\ve\geq -\frac{\ve}{2}\omega$, then
$S_\ve+\ve\omega$ is more singular than $T_\ve+\ve\omega$, and these two positive currents are cohomologous and have analytic singularities, and so
$$\int_{V_{\rm reg}} \left(S_\ve|_{V_{\rm reg}}+\frac{\ve}{2}\omega\right)_{\rm ac}^k\leq \int_{V_{\rm reg}} (S_\ve|_{V_{\rm reg}}+\ve\omega)_{\rm ac}^k\leq \int_{V_{\rm reg}} (T_\ve|_{V_{\rm reg}}+\ve\omega)_{\rm ac}^k,$$
thanks to \cite[Theorem 1.16]{BEGZ}. We note here that their result is for $V$ smooth, but these integrals don't change if we pass to a resolution of singularities, and that their result
is stated for the nonpluripolar product (instead of the Monge-Amp\`ere of the absolutely continuous part), but these agree for currents with analytic singularities.
We conclude that
$$\langle\alpha^{\dim V}\rangle_{X|V}\leq\sup_{S_\ve}\int_{V_{\rm reg}} \left(S_\ve|_{V_{\rm reg}}+\frac{\ve}{2}\omega\right)_{\rm ac}^k\leq \int_{V_{\rm reg}} (T_\ve|_{V_{\rm reg}}+\ve\omega)_{\rm ac}^k,$$
where the supremum is over all currents $S_\ve$ as before (with $\ve>0$ fixed), which proves \eqref{simp1}.

To prove \eqref{simp2}, we just note that by definition we have
$$\sup_{S_\ve} \int_{V_{\rm reg}} (S_\ve|_{V_{\rm reg}}+\ve\omega)_{\rm ac}^k\leq \langle\alpha^{\dim V}\rangle_{X|V}+\psi_V(\ve),$$
with $\psi_V(\ve)\to 0$ as $\ve\to 0$, where the supremum is over all currents $S_\ve\in[\alpha]$ with $S_\ve\geq -\ve\omega$ and with analytic singularities not containing $V$. Since the currents $T_\ve$ above are included in this supremum, we conclude that
$$\int_{V_{\rm reg}}(T_\ve|_{V_{\rm reg}}+\ve\omega)_{\rm ac}^k\leq \langle\alpha^{\dim V}\rangle_{X|V}+\psi_V(\ve),$$
and we are done.
\end{proof}
\begin{rem}\label{nonpl}
A very similar argument, using again \cite[Theorem 1.16]{BEGZ}, and using non-pluripolar products, can easily be used to prove \eqref{nonp}.
\end{rem}

Using this lemma we can prove the following Fujita-type approximation result for the numerical restricted volume (cf. \cite[Theorem 4.8]{Ma} in a more restrictive setting). In the algebraic setting, Fujita approximation results for the restricted volume were obtained independently in \cite{ELMNP,Ta}.

\begin{thm}\label{fujita} Let $(X,\omega)$ be a compact K\"ahler, and $[\alpha]$ a pseudoeffective class.
Then for every $\ve>0$ there is a  modification
$\mu_\ve:X_\ve\to X$, which is an isomorphism outside $E_{\rm nn}(\alpha)$ such that
$\mu_\ve^*(\alpha+\ve\omega)=A_\ve+E_\ve$ with $A_\ve$ a semipositive class and $E_\ve$ an effective $\mathbb{R}$-divisor, such that for every irreducible subvariety $V\not\subset E_{\rm nn}(\alpha)$ we have that the support of $E_\ve$ does not contain the proper transform $V_\ve$ of $V$, and
$$\int_{V_\ve}A_\ve^{\dim V}-\psi_V(\ve)\leq \langle\alpha^{\dim V}\rangle_{X|V}\leq \int_{V_\ve}A_\ve^{\dim V},$$
where $\psi_V(\ve)\to 0$ as $\ve\to 0$.
\end{thm}
By subtracting from $A_\ve$ a small multiple of the exceptional divisors of $\mu_\ve$, and adding this to $E_\ve$ one can also achieve that the class $A_\ve$ is K\"ahler, but we will not need this.
\begin{proof}
Given $\ve>0$ small, use Lemma \ref{simple} and obtain $T_\ve\in\alpha$ a closed $(1,1)$ current with analytic singularities contained in $E_{\rm nn}(\alpha)$, with $T_\ve\geq -\ve\omega$, and with
$$\langle\alpha^{\dim V}\rangle_{X|V}\leq \int_{V_{\rm reg}}(T_\ve|_{V_{\rm reg}}+\ve\omega)_{\rm ac}^{\dim V},$$
for all such $V$.
Let $\mu_\ve:X_\ve\to X$ be a principalization of the ideal sheaf of the singularities of $T_\ve$, followed by an embedded resolution of the singularities of the proper transform $V_\ve$ of $V$, which we therefore assume is smooth.
Then $\mu_\ve^*T_\ve$ has analytic singularities along an effective $\mathbb{R}$-divisor $E_\ve$ not containing $V_\ve$ and the Siu decomposition of $\mu_\ve^*(T_\ve+\ve\omega)$ is
$$\mu_\ve^*(T_\ve+\ve\omega)=\theta_\ve+[E_\ve],$$
where $\theta_\ve$ is a smooth closed form, which satisfies $\theta_\ve\geq 0$. Denote $A_\ve=[\theta_\ve]$. Since $\mu_\ve$ is an isomorphism at the generic point of $V$, as in Lemma \ref{finite} we see that
$$\int_{V_{\rm reg}}(T_\ve|_{V_{\rm reg}}+\ve\omega)_{\rm ac}^k=\int_{V_\ve}(\mu_\ve^*(T_\ve+\ve\omega)|_{V_\ve})_{\rm ac}^k,$$
where we set $k=\dim V$, and also
$$\int_{V_\ve}(\mu_\ve^*(T_\ve+\ve\omega)|_{V_\ve})_{\rm ac}^k=\int_{V_\ve}((\theta_\ve+[E_\ve])|_{V_\ve})_{\rm ac}^k=\int_{V_\ve}\theta_\ve^k=\int_{V_\ve}A_\ve^k.$$
Hence we conclude that
$$\langle\alpha^{\dim V}\rangle_{X|V}\leq \int_{V_\ve}A_\ve^k,$$
which is half of what we want.
The other half follows as in the proof of \eqref{simp2}, which gives
$$\int_{V_\ve}A_\ve^k=\int_{V_{\rm reg}}(T_\ve|_{V_{\rm reg}}+\ve\omega)_{\rm ac}^k\leq \langle\alpha^{\dim V}\rangle_{X|V}+\psi_V(\ve),$$
and we are done.
\end{proof}

This in turn implies the following log concavity result (cf. \cite{ELMNP,Ma}):
\begin{thm}\label{concave}
If $\alpha_1,\alpha_2$ are two pseudoeffective classes on $X$ with $V$ an irreducible $k$-dimensional suvbariety not contained in $E_{\rm nn}(\alpha_1)\cup E_{\rm nn}(\alpha_2)$, then
$$\langle(\alpha_1+\alpha_2)^k\rangle_{X|V}^{\frac{1}{k}}\geq \langle \alpha_1^k\rangle_{X|V}^{\frac{1}{k}}+\langle \alpha_2^k\rangle_{X|V}^{\frac{1}{k}}.$$
\end{thm}
\begin{proof}
Since $V$ is irreducible, $V$ not contained in $E_{\rm nn}(\alpha_1)\cup E_{\rm nn}(\alpha_2)$ is equivalent to $V$ not contained in $E_{\rm nn}(\alpha_1)$ and also not contained in
$E_{\rm nn}(\alpha_2)$.
Note also that
$$E_{\rm nn}(\alpha_1+\alpha_2)\subset E_{\rm nn}(\alpha_1)\cup E_{\rm nn}(\alpha_2).$$
So all the three numerical restricted volumes are defined by \eqref{voldef}.
Fix $\ve>0$, and apply Theorem \ref{fujita} to $[\alpha_1]$ to get a modification $\mu_1:X_1\to X$ and a current $T_\ve^1\in[\alpha_1]$ with analytic singularities not containing $V$, $T_\ve^1\geq -\ve\omega$, with
$\mu_1^*(T^1_\ve+\ve\omega)=\theta_1+[E_1]$ and
$$\left|\int_{V_1}\theta_1^k - \langle \alpha_1^k\rangle_{X|V}\right|\leq \psi(\ve),$$
where $V_1$ is the proper transform of $V$ and $\psi(\ve)\to 0$ as $\ve\to 0$.
We do the same for $[\alpha_2]$ and get $\mu_2:X_2\to X$, and $T^2_\ve$ with
$$\left|\int_{V_2}\theta_2^k - \langle \alpha_2^k\rangle_{X|V}\right|\leq \psi(\ve).$$
We can pass to a common resolution $\mu:\ti{X}\to X$, which is still an isomorphism at the generic point of $V$, and pullback everything upstairs, without changing notation, so that $\mu^*(T^1_\ve+\ve\omega)=\theta_1+[E_1]$, and so on.

Then $T^1_\ve+T^2_\ve$ is a current in $[\alpha_1+\alpha_2]$ with analytic singularities not containing $V$, $T_\ve^1+T_\ve^2\geq -2\ve\omega$, and
$$\mu^*(T^1_\ve+T^2_\ve+2\ve\omega)=\theta_1+\theta_2+[E_1]+[E_2],$$
and
$$\langle(\alpha_1+\alpha_2)^k\rangle_{X|V}\geq \int_{V_{\rm reg}} ((T^1_\ve+T^2_\ve+2\ve\omega)|_{V_{\rm reg}})_{\rm ac}^k-\psi(\ve)=\int_{\ti{V}} (\theta_1+\theta_2)^k-\psi(\ve),$$
for a possibly different function $\psi(\ve)$.
By the usual log concavity property of the volume of a nef class \cite{Bo} we have
$$\left(\int_{\ti{V}} (\theta_1+\theta_2)^k\right)^{\frac{1}{k}}\geq \left(  \int_{\ti{V}}\theta_1^k \right)^{\frac{1}{k}}+\left( \int_{\ti{V}}\theta_2^k\right)^{\frac{1}{k}}.$$
Putting these together and letting $\ve\to 0$ finishes the proof.
\end{proof}

For fixed $V$, the set $\mathcal{A}$ of big classes $\alpha$ with $V\not\subset E_{\rm nK}(\alpha)$ is easily seen to be convex and open (see \cite[Proposition 4.10]{Ma}). As in  \cite[Corollary 4.11]{Ma} we obtain:
\begin{cor}
If $V$ is an irreducible subvariety of $X$, and $\mathcal{A}$ is the open convex cone of $V$-big classes, then the function $\alpha\mapsto \langle\alpha^{\dim V}\rangle_{X|V}$ from $\mathcal{A}$ to $\mathbb{R}$ is continuous.
\end{cor}
\begin{proof}
Indeed the function $\alpha\mapsto \langle\alpha^{\dim V}\rangle_{X|V}^{\frac{1}{k}}$ is concave on $\mathcal{A}$ thanks to Theorem \ref{concave}. Since $\mathcal{A}$ is convex and open,
this function is therefore continuous.
\end{proof}

\begin{ex}\label{exx2}
In general, the function $\alpha\mapsto \langle\alpha^{\dim V}\rangle_{X|V}$ is not continuous on the whole big cone. An example of such a discontinuity for the algebraic restricted volume was given in \cite[Example 5.10]{ELMNP}, and we now check that in this same example the numerical restricted volume is discontinuous as well.
Of course, Proposition \ref{con} says that (assuming Conjecture \ref{con3}) the numerical restricted volume is continuous at certain points on the closure of $\mathcal{A}$ (this closure is of course the closed convex cone of $V$-psef classes), the ones where $V$ is one of the irreducible components of $E_{\rm nK}(\alpha)$, and indeed that its value there is zero.

Following \cite[Example 5.10]{ELMNP} we let $\pi:X\to\mathbb{CP}^3$ be the blowup of $\mathbb{CP}^3$ along a line $\ell$, and let $L=\pi^*\mathcal{O}(1)$ and $V\subset X$ be a smooth curve of bidegree $(2,1)$ inside the exceptional divisor $E\cong\mathbb{CP}^1\times\mathbb{CP}^1$, where $\pi|_E:E\to\ell$ is the projection onto the first factor. It is computed in \cite{ELMNP} that $\vol_{X|V}(L)=1,$ while $\int_Vc_1(L)=2$ and $\vol_{X|V}(L-\frac{1}{m}E)=2+\frac{1}{m}$ for all $m$ large. The $\mathbb{Q}$-divisors $L-\frac{1}{m}E$ are ample for $m$ large, and so Theorem \ref{hisamat} gives $\langle c_1(L-\frac{1}{m}E)\rangle_{X|V}=\vol_{X|V}(L-\frac{1}{m}E)=2+\frac{1}{m}.$ Also, since $L$ is nef, Lemma \ref{l4} gives $\langle c_1(L)\rangle_{X|V}=\int_Vc_1(L)=2.$

On the other hand, if we fix $m$ large so that $H=L-\frac{1}{m}E$ is ample, then we have
$$E=E_{\rm nK}(c_1(L))=\bigcap_{\ve>0} E_{\rm nn}(c_1(L-\ve H))=\bigcap_{\ve>0}E_{\rm nn}\left(c_1\left(L+\frac{\ve}{m(1-\ve)}E\right)\right),$$
and so $E\subset E_{\rm nn}(c_1(L+\ve E))$ for all $\ve>0$. Since $V\subset E$, this gives $\langle c_1(L+\ve E)\rangle_{X|V}=0$ for all $\ve>0$, and so the function
$t\mapsto \langle c_1(L+tE)\rangle_{X|V}$ is discontinuous at $t=0$.
\end{ex}

\section{Restricted volumes and Zariski decompositions}\label{sect6}
Let us briefly recall a few facts about Zariski decompositions on surfaces \cite{Za}. A good reference for all the unproved statements which follow is \cite{BKS}. Let $L$ be a pseudoeffective line bundle on a smooth projective surface $X$. Then $L$ can be written uniquely as $L=P+N$ where $P$ is a nef $\mathbb{Q}$-divisor class, $N=\sum_j a_jD_j, a_j\in\mathbb{Q}_+$ is an effective $\mathbb{Q}$-divisor, $P\cdot N=0$, and either $N=0$ or the matrix $(D_i\cdot D_j)$ is negative definite. Furthermore we have $H^0(X,mL)\cong H^0(X,mP)$ for any $m\geq 0$ such that $mP$ is integral. It follows that $\vol(L)=\vol(P)=P^2$.
We also have
$$\mathbb{B}_+(L)=\mathbb{B}_+(P)=\mathrm{Null}(P),$$
while
$$\mathbb{B}_-(L)=\mathrm{Supp}(N).$$

A substitute for this theory in higher dimensions are divisorial Zariski decompositions, introduced by Boucksom \cite{Bo2} (and Nakayama \cite{Na} in the algebraic case).
Let $X$ be a compact K\"ahler manifold and $[\alpha]$ a big class. The divisorial Zariski decomposition $[\alpha]=P+N$ of $[\alpha]$ is defined by letting
$$N=\sum_D \nu([\alpha],D)[D],$$
where we sum over all prime divisors $D\subset X$, and defining $P=[\alpha]-N$. Here $\nu([\alpha],D)=\inf_{x\in D}\nu(T_{min},x)$, where $T_{min}$ is any positive current with minimal singularities in the class $[\alpha]$.
Boucksom \cite{Bo2} shows that $N$ is an effective $\mathbb{R}$-divisor (possibly zero), and that $P$ is a big class with $\vol(\alpha)=\vol(P)$. By construction $P$ is nef in codimension $1$, which means
that $\nu(P,D)=0$ for all prime divisors $D$ (equivalently, every irreducible component of $E_{\rm nn}(P)$ has codimension at least $2$). Clearly, the support of $N$ coincides with the set of codimension $1$ components of $E_{\rm nn}(\alpha)$.

If it happens that $P$ is nef, then we say that $[\alpha]$ has a Zariski decomposition. These do not always exist (see \cite{Bo2}), but they always exist on K\"ahler surfaces (because in this case nef in codimension $1$ is the same as nef). 

We need a few preparatory lemmas. The first one is a straightforward generalization of \cite[Claim 4.7]{Ma}.

\begin{lem}\label{l1}
If $[\alpha]$ is big and $[\alpha]=P+N$ is its divisorial Zariski decomposition, then we have
$$E_{\rm nK}(\alpha)=E_{\rm nK}(P).$$
\end{lem}
\begin{proof}
By slight abuse of notation, let us also set $\mathcal{N}=\sum_D \nu([\alpha],D)[D]$, which now is a closed positive current in the cohomology class $N$. Fix $\delta\geq 0$.
We claim that the map $$\mathcal{P}:T\mapsto T-\mathcal{N},$$
gives a bijection between the set of closed positive $(1,1)$ currents $T$ in $[\alpha]$ which satisfy $T\geq\delta\omega$, and closed positive $(1,1)$ currents $\ti{T}$ in $P$ which satisfy $\ti{T}\geq\delta\omega$.
Indeed, if $T\in [\alpha]$ is a closed $(1,1)$ current with $T\geq\delta\omega$,
for any prime divisor $D$ we have $\nu(T,D)\geq\nu(T_{min},D)=\nu([\alpha],D)$, since $T_{min}$ is less singular than $T$.
On the other hand, we have the Siu decomposition
$$T=S+\sum_D\nu(T,D)[D],$$
which has the property that $S\geq\delta\omega$ too. Therefore
$$T-\mathcal{N}=S+\sum_D(\nu(T,D)-\nu([\alpha],D))[D]\geq S\geq\delta\omega,$$
is a closed $(1,1)$ current in the class $P$. Conversely if $\ti{T}$ is a closed $(1,1)$ current in $P$ with $\ti{T}\geq\delta\omega$, then clearly $\ti{T}+\mathcal{N}\geq \delta\omega$ is a closed $(1,1)$ current in $[\alpha]$.

Also $\mathcal{P}$ clearly sends K\"ahler currents with analytic singularities in $[\alpha]$ to K\"ahler currents with analytic singularities in $P$, and vice versa.
Recall that $$E_{\rm nK}(\alpha)=\bigcap_{T} E_+(T),$$
where the intersection is over all K\"ahler currents $T\in[\alpha]$ with analytic singularities, and similarly for $P$.
If $x\not\in E_{\rm nK}(\alpha)$ then we can find a K\"ahler current $T\in [\alpha]$ with analytic singularities which is smooth near $x$. This implies that $x$ is not in the support of $\mathcal{N}$, and so $T-\mathcal{N}$ is a K\"ahler current with analytic singularities
in $P$, smooth near $x$. This shows that $E_{\rm nK}(P)\subset E_{\rm nK}(\alpha)$.

On the other hand, if $\ti{T}\in P$ is a K\"ahler current with analytic singularities, we
claim that $\nu(\ti{T},D)>0$ for any prime divisor $D$ with $\nu([\alpha],D)>0$. If not, then there would be a prime divisor $D$ with $\nu([\alpha],D)>0$ but $\ti{T}$ is smooth at the generic point of $D$. We have
$\ti{T}\geq\ve\omega$ for some $\ve>0$. Pick $\eta>0$ small enough so that the class $\eta N$ has a smooth representative $\chi$ with $\chi\geq -\frac{\ve}{2}\omega$. Then
$$S=\ti{T}+\chi+(1-\eta)\mathcal{N}\geq \frac{\ve}{2}\omega,$$
is a K\"ahler current with analytic singularities in $[\alpha]$ and
$$\nu(S,D)=(1-\eta)\nu([\alpha],D)<\nu([\alpha],D),$$
which is a contradiction.

Let now $x\not\in E_{\rm nK}(P)$, so there is a K\"ahler current $\ti{T}\in P$ with analytic singularities, smooth near $x$. We have just proved that $\nu(\ti{T},D)>0$ for any prime divisor $D$ in the support of $\mathcal{N}$, hence $x\not\in\mathrm{Supp}(N)$, and so $$T=\mathcal{P}^{-1}(\ti{T})=\ti{T}+\mathcal{N},$$
is a K\"ahler current in $[\alpha]$ with analytic singularities and smooth near $x$, which proves $E_{\rm nK}(\alpha)\subset E_{\rm nK}(P).$
\end{proof}

The following is an improvement on \cite[Proposition 4.6]{Ma}.

\begin{lem}\label{l3}
If $[\alpha]$ is big and $V$ is an irreducible positive dimensional subvariety of $X$ which is not contained in $E_{\rm nn}(\alpha)$, then we have
$$\langle\alpha^{\dim V}\rangle_{X|V}=\langle P^{\dim V}\rangle_{X|V}.$$
\end{lem}
\begin{proof}
Write $[\alpha]=P+N$ for the divisorial Zariski decomposition.
Since $V$ is not contained in $E_{\rm nn}(\alpha)$, it is also not contained in $\mathrm{Supp}(N)$, since this is a subset of $E_{\rm nn}(\alpha)$.

We need a variant of the construction in the previous lemma. For $\ve>0$ let $T_{min,\ve}$ be a positive current with minimal singularities in the class $[\alpha+\ve\omega]$, and let
$\nu_\ve(\alpha,x)=\nu(T_{min,\ve},x)$. Then we have that $\nu_\ve(\alpha,x)$ increases to $\nu([\alpha],x)$ as $\ve$ decreases to zero. As before, let $\nu_\ve(\alpha,D)=\inf_{x\in D}\nu_\ve(\alpha,x)$ and define
a closed positive current $\mathcal{N}_\ve=\sum_D \nu_\ve(\alpha,D)[D]$. This sum converges, since it is dominated by $\mathcal{N}$ because $\nu_\ve(\alpha,D)\leq \nu([\alpha],D)$, which also implies that $\mathcal{N}_\ve$ is the current of integration on an effective $\mathbb{R}$-divisor, i.e. only finitely many terms $\nu_\ve(\alpha,D)[D]$ are nonzero. Of course if $\nu_\ve(\alpha,D)>0$, then $\nu([\alpha],D)>0$ too, so let $D_1,\dots, D_k$ be all the prime divisors in $X$ with $\nu([\alpha],D)>0$.
Fix $\beta_1,\dots,\beta_k$ closed smooth forms on $X$ which represent $[D_1],\dots,[D_k]$.

First of all, if $\ti{T}\in P$ is a closed $(1,1)$ current with analytic singularities and with $\ti{T}\geq -\ve\omega$ and $V\not\subset E_+(\ti{T})$, then $\ti{T}+\mathcal{N}\geq-\ve\omega$ is a closed $(1,1)$ current with analytic singularities in $[\alpha]$ with singular set not containing $V$, and
$$\int_{V_{\rm reg}} (\ti{T}|_{V_{\rm reg}}+\ve\omega)_{\rm ac}^{\dim V}=\int_{V_{\rm reg}}((\ti{T}+\mathcal{N})|_{V_{\rm reg}}+\ve\omega)_{\rm ac}^{\dim V},$$
since $(\mathcal{N}|_{V_{\rm reg}})_{\rm ac}=0$. This shows that $\langle P^{\dim V}\rangle_{X|V}\leq \langle \alpha^{\dim V}\rangle_{X|V}$.

If now $T\in[\alpha]$ is a closed $(1,1)$ current with analytic singularities, with $T\geq -\ve\omega$ and $V\not\subset E_+(T)$, then as in the proof of Lemma \ref{l1} we have that $T-\mathcal{N}_\ve\geq-\ve\omega$ and so
$$\ti{T}=T-\mathcal{N}_\ve+\sum_{j=1}^k \left(\nu_\ve(\alpha,D_j)-\nu([\alpha],D_j)\right)\beta_j\geq-\ve\omega-\psi(\ve)\omega,$$
is a closed $(1,1)$ current in $P$ with analytic singularities and with singular set not containing $V$, where $\psi(\ve)\to 0$ as $\ve\to 0$ (and is independent of $T$), since $\nu_\ve(\alpha,D_j)\to\nu([\alpha], D_j)$. Furthermore,
\[\begin{split}
\int_{V_{\rm reg}} (T|_{V_{\rm reg}}+\ve\omega)_{\rm ac}^{\dim V}&=\int_{V_{\rm reg}} ((T-\mathcal{N}_\ve)|_{V_{\rm reg}}+\ve\omega)_{\rm ac}^{\dim V}\\
&=\int_{V_{\rm reg}} (\ti{T}|_{V_{\rm reg}}+(\ve+\psi(\ve))\omega)_{\rm ac}^{\dim V}+\Psi(\ve),
\end{split}\]
using Lemma \ref{finite}, where $\Psi(\ve)\to 0$ as $\ve\to 0$ (and is independent of $T$). This implies $\langle P^{\dim V}\rangle_{X|V}\geq \langle \alpha^{\dim V}\rangle_{X|V}$.
\end{proof}

We can now give the proof of Theorem \ref{l5}:
\begin{proof}[Proof of Theorem \ref{l5}]
If $V\subset E_{\rm nn}(\alpha)$, then we have $\langle\alpha^{\dim V}\rangle_{X|V}=0$ by definition, so we may assume that $V$ is not contained in $E_{\rm nn}(\alpha)$. Since we assume that $P$ is nef, Lemmas \ref{l3} and \ref{l4} give
$$\langle\alpha^{\dim V}\rangle_{X|V}=\langle P^{\dim V}\rangle_{X|V}=\int_V P^{\dim V}.$$
Thanks to Lemma \ref{l1} we have $E_{\rm nK}(\alpha)=E_{\rm nK}(P),$ so that $V$ is one of the irreducible components of $E_{\rm nK}(P)$. The main result of \cite{CT} then shows that
$$\int_V P^{\dim V}=0.$$
\end{proof}

In fact, Theorem \ref{l5} holds more generally if $\alpha$ admits a Zariski decomposition on some bimeromorphic model.

\begin{thm}\label{l7}
If $[\alpha]$ is big and if there exists $\mu:\ti{X}\to X$ a modification, with $\ti{X}$ K\"ahler and such that $[\mu^*\alpha]$ admits a Zariski decomposition, and $V$ is one of the irreducible components of $E_{\rm nK}(\alpha)$ which we assume is not contained in $\mu(\mathrm{Exc}(\mu))$, then
$$\langle\alpha^{\dim V}\rangle_{X|V}=0.$$
\end{thm}

Note that the assumptions of this theorem hold for example if we assume that there exists a positive current $T_{min}$ in $[\alpha]$ with minimal singularities which has analytic singularities, by \cite[Proposition 4.1]{Ma3} (although it is stated there for divisors, the simple proof works for general $(1,1)$ classes). In this case, the map $\mu$ is a resolution of the singularities of $T_{min}$, which are along $E_{\rm nn}(\alpha)=\mu(\mathrm{Exc}(\mu))$, and we may assume without loss of generality that $V\not\subset E_{\rm nn}(\alpha)$.

\begin{proof}
Since $V$ is not contained in $\mu(\mathrm{Exc}(\mu))$, then Lemma \ref{invar}
gives $\langle\alpha^k\rangle_{X|V}=\langle\mu^*\alpha^k\rangle_{\ti{X}|\ti{V}}$, where $\ti{V}$ is the proper transform of $V$ under $\mu$, and $k=\dim V$. We also have that
$E_{\rm nK}(\mu^*\alpha)=\mu^{-1}(E_{\rm nK}(\alpha))\cup\mathrm{Exc}(\mu)$, thanks to Lemma \ref{pull}, and so $\ti{V}$ is an irreducible component of
$E_{\rm nK}(\mu^*\alpha)$.
The result is trivial if $V\subset E_{\rm nn}(\alpha)$, so we may assume that $V\not\subset E_{\rm nn}(\alpha)$, and thanks to Lemma \ref{nonnef} we have that
$\ti{V}\not\subset E_{\rm nn}(\mu^*\alpha)$.
If $P$ denotes the positive part of $[\mu^*\alpha]$ in the divisorial Zariski decomposition, by assumption we have that $P$ is nef, and Lemmas \ref{l3} and \ref{l4} give
$$\langle\mu^*\alpha^k\rangle_{\ti{X}|\ti{V}}=\langle P^k\rangle_{\ti{X}|\ti{V}}=\int_{\ti{V}} P^k.$$
Thanks to Lemma \ref{l1} we have $E_{\rm nK}(\mu^*\alpha)=E_{\rm nK}(P),$ so that $\ti{V}$ is one of the irreducible components of $E_{\rm nK}(P)$. The main result of \cite{CT} then shows that
$$\int_{\ti{V}} P^k=0.$$
\end{proof}

\section{Moving Seshadri constants}\label{sect7}
Recall that if $[\alpha]$ is a nef $(1,1)$ class and $x\in X$ we let the Seshadri constant of $[\alpha]$ at $x$ be
$$\ve(\alpha,x)=\sup\{ \lambda\geq 0\ |\ \pi^*[\alpha]-\lambda [E]\ {\rm nef}\},$$
where $\pi:\ti{X}\to X$ is the blowup of $X$ at $x$, and $E$ is the exceptional divisor. These were first introduced by Demailly \cite{Dem}.
It follows easily from the definition that $\ve(\cdot,x)$ is a continuous function on the nef cone. Furthermore, it is also concave since
$\ve(\alpha+\beta,x)\geq \ve(\alpha,x)+\ve(\beta,x)$ for any nef classes $[\alpha],[\beta]$, and it satisfies \eqref{subvar} as shown in \cite[Theorem 2.8]{To} (cf. \cite[Proposition 5.1.9]{Laz} in the algebraic case) using crucially the Demailly-P\u{a}un Theorem \cite{DP}. It is also not hard to see that (see \cite{Dem}) if $[\alpha]$ is K\"ahler then we have
\begin{equation}\label{seseq3}
\ve(\alpha,x)=\sup\left\{
     \gamma\geq 0\ \bigg|\ \begin{array}{ll}&\exists T\in[\alpha], T\geq 0, T \textrm{ has an isolated}\\
& \textrm{ singularity at }x \textrm{ with } \nu(T,x)\geq\gamma \end{array}\right\}.
\end{equation}

The moving Seshadri constants are a generalization of this concept to classes which need not be nef. Let $[\alpha]$ be a pseudoeffective class on a compact K\"ahler manifold $X$. Given any point $x\in X$, following \cite{Nak2, ELMNP} in the algebraic case, we define the moving Seshadri constant $\ve(\|\alpha\|,x)$ as follows. If $x\in E_{\rm nK}(\alpha)$ we set $\ve(\|\alpha\|,x)=0$, and otherwise we set
$$\ve(\|\alpha\|,x)=\sup_{\mu^*[\alpha]=[\beta]+[E]}\ve(\beta,\mu^{-1}(x)),$$
where the supremum is over all modifications $\mu:\ti{X}\to X$, which are isomorphisms near $x$, and over all decompositions
$\mu^*[\alpha]=[\beta]+[E]$ where $[\beta]$ is a K\"ahler class and $E$ is an effective $\mathbb{R}$-divisor which does not contain $\mu^{-1}(x)$. Clearly, when $[\alpha]=c_1(D)$ for a pseudoeffective $\mathbb{R}$-divisor $D$, we have that $\ve(\|\alpha\|,x)=\ve(\|D\|,x),$ as defined in \cite[Section 6]{ELMNP}.

Note that such decompositions always exist when $x\not\in E_{\rm nK}(\alpha)$, since we may pick a K\"ahler current $T$ in the class $[\alpha]$ which has analytic singularities and is smooth near $x$, and let $\mu$ be the resolution of the singularities of $T$, and all the stated properties hold. More precisely, say we have that $T\geq \ve\omega$, and we resolve so that
$$\mu^*T=\theta+[F],$$
where $\theta\geq\ve\mu^*\omega$ is a smooth form and $F$ an effective $\mathbb{R}$-divisor, and $\mu$ is a composition of blowups with smooth centers. But recall that if $G$ is the exceptional set of $\mu$ (which is an effective divisor) then there is $\delta>0$ small so that
$[\theta]-\delta[G]$ is a K\"ahler class on $\ti{X}$. Then we let $[\beta]=[\theta]-\delta[G]$ and $E=F+\delta G$.

Note also that if $\mu$ is as above then we have $\ve(\beta,\mu^{-1}(x))\leq \vol(\beta)^{\frac{1}{n}},$ thanks to \eqref{subvar},
and also
$$\vol(\beta)=\int_{\ti{X}}\beta^n\leq \vol(\beta+[E])=\vol(\mu^*\alpha)=\vol(\alpha),$$
where the inequality follows from the definition (and the fact that $E$ is an effective $\mathbb{R}$-divisor). Together, these imply that
$$\ve(\|\alpha\|,x)\leq \vol(\alpha)^{\frac{1}{n}},$$
and so the supremum in the definition of the moving Seshadri constant is finite.

First, let us observe the following:
\begin{prop}If $[\alpha]$ is a nef class, then for every $x\in X$ we have
$$\ve(\|\alpha\|,x)=\ve(\alpha,x).$$
\end{prop}
\begin{proof}
If $x\in E_{\rm nK}(\alpha)$ then $\ve(\|\alpha\|,x)=0$ by definition, while $\ve(\alpha,x)=0$ because of \eqref{subvar} together with the main theorem of \cite{CT}. Assume now that $x\not\in E_{\rm nK}(\alpha)$.
If $\mu:\ti{X}\to X$ is any modification as in the definition above, so that $\mu^*[\alpha]=[\beta]+[E]$ where $[\beta]$ is a K\"ahler class and $E$ is an effective $\mathbb{R}$-divisor which does not contain $\mu^{-1}(x)$,
then note that
\begin{equation}\label{tech}\begin{split}\ve(\alpha,x)&=\inf_{V\ni x}\left(\frac{\int_{V}\alpha^{\dim V}}{\mathrm{mult}_x V}\right)^{\frac{1}{\dim V}}=\inf_{\ti{V}\ni \mu^{-1}(x)}\left(\frac{\int_{\ti{V}}\mu^*\alpha^{\dim V}}{\mathrm{mult}_x V}\right)^{\frac{1}{\dim V}}\\
&\geq \inf_{\ti{V}\ni \mu^{-1}(x)}\left(\frac{\int_{\ti{V}}\beta^{\dim V}}{\mathrm{mult}_x V}\right)^{\frac{1}{\dim V}}=\ve(\beta,\mu^{-1}(x)),\end{split}\end{equation}
using \eqref{subvar} and the fact that, since $E$ is an effective $\mathbb{R}$-divisor which does not contain $\ti{V}$, and $\mu^*[\alpha]$ is nef and $[\beta]$ is K\"ahler, we have
$$\int_{\ti{V}} \mu^*\alpha^{\dim V}=\vol(\mu^*\alpha|_{\ti{V}})\geq \vol(\beta|_{\ti{V}})=\int_{\ti{V}} \beta^{\dim V},$$
as above.
This gives $\ve(\alpha,x)\geq \ve(\|\alpha\|,x)$. Conversely, since $x\not\in E_{\rm nK}(\alpha)$, we can find a K\"ahler current in the class $[\alpha]$ with analytic singularities and smooth near $x$. Resolving it we obtain a modification $\mu:\ti{X}\to X$, which is an isomorphism near $x$, and such that
$\mu^*[\alpha]=[\beta]+[E]$ where $[\beta]$ is a K\"ahler class and $E$ is an effective $\mathbb{R}$-divisor which does not contain $\mu^{-1}(x)$. For any $k\geq 2$ we obtain
$$\mu^*[\alpha]=\left(\left(1-\frac{1}{k}\right)\mu^*[\alpha]+\frac{1}{k}[\beta]\right)+\frac{1}{k}[E]=:[\beta_k]+\frac{1}{k}[E],$$
which is the sum of the K\"ahler class $[\beta_k]$ and an effective $\mathbb{R}$-divisor which does not contain $\mu^{-1}(x)$. By definition,
$$\ve(\|\alpha\|,x)\geq \ve(\beta_k,\mu^{-1}(x)),$$
and letting $k\to\infty$ (recall that $\ve(\cdot,\mu^{-1}(x))$ is continuous on the nef cone) we obtain
$$\ve(\|\alpha\|,x)\geq \ve(\mu^*\alpha,\mu^{-1}(x))=\ve(\alpha,x),$$
where the last equality follows from \eqref{subvar} as in the first line of \eqref{tech}.
\end{proof}

We first prove half of Theorem \ref{movable2}:

\begin{thm}\label{movable} Assume Conjecture \ref{con3}.
Given any pseudoeffective class $[\alpha]$ and any $x\in X$ we have
$$\ve(\|\alpha\|,x)\leq\inf_{V\ni x}\left(\frac{\langle \alpha^{\dim V}\rangle_{X|V}}{\mathrm{mult}_x V}\right)^{\frac{1}{\dim V}},$$
where the infimum is over all irreducible analytic subvarieties $V$ containing $x$.
\end{thm}
\begin{proof}
If $x\in E_{\rm nK}(\alpha)$ then both sides are zero, using Conjecture \ref{con3}. If $x\not\in E_{\rm nK}(\alpha)$,
then any irreducible subvariety $V$ which contains $x$ satisfies $V\not\subset E_{\rm nK}(\alpha)$. Thanks to Lemma \ref{volr}, we have that
$$\langle\alpha^k\rangle_{X|V}=\sup_{T}\int_{V_{\rm reg}}(T|_{V_{\rm reg}})_{\rm ac}^k,$$
where $k=\dim V$ and the supremum is over all K\"ahler currents $T\in[\alpha]$ with analytic singularities which do not contain $V$.

In fact, we claim that we obtain the same result if we take the supremum only among all K\"ahler currents $T\in[\alpha]$ with analytic singularities which do not contain $V$, and which are smooth near $x$. Indeed, since $x\not\in E_{\rm nK}(\alpha)$, there exists a K\"ahler current $S\in[\alpha]$ with analytic singularities which do not contain $V$ and which is smooth near $x$. If $T$ is any other K\"ahler current in the class $[\alpha]$ with analytic singularities which do not contain $V$, then writing $T=\alpha+\ddbar\vp_1$ and $S=\alpha+\ddbar\vp_2$ we have that $\hat{T}=\alpha+\ddbar\max(\vp_1,\vp_2)$ is also a K\"ahler current in the class $[\alpha]$, which restricts to $V$, is locally bounded near $x$, and is globally less singular than $T$. Then regularizing $\hat{T}$ we obtain another K\"ahler current $\ti{T}$ in the class $[\alpha]$, with analytic singularities which do not contain $V$ and which is smooth near $x$, and which is less singular than $T$. Therefore \cite[Theorem 1.16]{BEGZ} implies that
\begin{equation}\label{equal2}
\int_{V_{\rm reg}}(T|_{V_{\rm reg}})_{\rm ac}^k\leq \int_{V_{\rm reg}}(\ti{T}|_{V_{\rm reg}})_{\rm ac}^k,
\end{equation}
which proves the claim.

Let now $T$ be any K\"ahler current in $[\alpha]$ with analytic singularities which do not contain $V$, and which is smooth near $x$, and let $\mu:\ti{X}\to X$ be a resolution of its singularities (which is therefore an isomorphism near $x$), so $\mu^*T=\theta+[F]$ where $\theta$ is smooth and semipositive and $F$ is an effective $\mathbb{R}$-divisor which does not contain $\mu^{-1}(x)$. As before, if $G$ denotes the exceptional set of $\mu$ then for all $\delta>0$ sufficiently small we have that $[\theta]-\delta[G]=:[\beta_\delta]$ is a K\"ahler class on $\ti{X}$, and $E_\delta:=F+\delta G$ is still an effective $\mathbb{R}$-divisor which does not contain $\mu^{-1}(x)$. Then $\mu^*[\alpha]=[\beta_\delta]+[E_\delta]$ and we have
\begin{equation}\label{equal}
\int_{V_{\rm reg}}(T|_{V_{\rm reg}})_{\rm ac}^k=\int_{\ti{V}} \theta^k=\int_{\ti{V}} (\beta_\delta+\delta c_1(G))^k,
\end{equation}
where $\ti{V}$ denotes the proper transform of $V$.
Take now a sequence $T_j$ of K\"ahler currents as before so that
$$\lim_{j\to\infty}\int_{V_{\rm reg}}(T_j|_{V_{\rm reg}})_{\rm ac}^k=\langle\alpha^k\rangle_{X|V},$$
which exist thanks to our earlier claim.
Resolve $T_j$ by $\mu_j:X_j\to X$ with $\mu_j^*T_j=\theta_j+[F_j]$ and let as before $[\beta_{j,\delta}]=[\theta_j]-\delta[G_j]$.
For each $j$ choose $\delta_j$ small enough so that $[\beta_j]:=[\beta_{j,\delta_j}]$ is K\"ahler and
$$\int_{\ti{V}_j} (\beta_j+\delta_j c_1(G_j))^k\leq \int_{\ti{V}_j} \beta_j^k+\frac{1}{j}.$$
This gives
$$\int_{V_{\rm reg}}(T_j|_{V_{\rm reg}})_{\rm ac}^k\leq \int_{\ti{V}_j} \beta_j^k+\frac{1}{j},$$
and since $\mu_j^*[\alpha]=[\beta_j]+[E_j]$ with $\beta_j$ K\"ahler and $E_j$ an effective $\mathbb{R}$-divisor which does not contain $\mu_j^{-1}(x)$, then
$$\int_{V_{\rm reg}}(T_j|_{V_{\rm reg}})_{\rm ac}^k\leq \sup_{\mu}\int_{\ti{V}} \beta^k+\frac{1}{j},$$
and passing to the limit in $j$ we have proved that
$$\langle\alpha^k\rangle_{X|V}=\sup_{T}\int_{V_{\rm reg}}(T|_{V_{\rm reg}})_{\rm ac}^k\leq \sup_{\mu}\int_{\ti{V}} \beta^k, $$
where the supremum is over all modifications $\mu$ which are isomorphism near $x$ and so that $\mu^*[\alpha]=[\beta]+[E]$ with $[\beta]$ K\"ahler and $[E]$ is an effective $\mathbb{R}$-divisor which does not contain $\mu^{-1}(x)$.

Conversely suppose we are given a modification $\mu$ which is an isomorphism near $x$ and so that $\mu^*[\alpha]=[\beta]+[E]$ with $[\beta]$ K\"ahler and $[E]$ is an effective $\mathbb{R}$-divisor which does not contain $\mu^{-1}(x)$. Choosing a K\"ahler metric $\omega\in[\beta]$, we have that $T=\mu_*(\omega+[E])$ is a K\"ahler current in the class $[\alpha]$ which restricts to $V$ and is smooth near $x$. Note that for any closed positive current $S$ on $\ti{X}$ we have that $\mu^*\mu_*S-S$ is supported on $E=\mathrm{Exc}(\mu)$, so by Federer's support theorem (see e.g. \cite{Siu74}) we have $\mu^*\mu_*S-S=[F]$ where $F$ is an $\mathbb{R}$-divisor with support contained in the support of $E$. We conclude that $(\mu^*T)_{\rm ac}=(\omega+[E])_{\rm ac}=\omega,$ and so
$$\int_{V_{\rm reg}}(T|_{V_{\rm reg}})_{\rm ac}^k=\int_{\ti{V}} \beta^k,$$
holds. Regularizing $T$ we obtain a sequence of K\"ahler currents $\ti{T}_j$ in the class $\alpha$, with analytic singularities which do not contain $V$ and which are smooth near $x$, and
with $(\ti{T}_j)_{\rm ac}(y)\to T_{\rm ac}(y)$ pointwise for all $y$ where $T(y)$ is smooth (in particular, for all generic $y\in V$). Thanks to Fatou's lemma we have
$$\liminf_{j\to\infty}\int_{V_{\rm reg}}(\ti{T}_j|_{V_{\rm reg}})_{\rm ac}^k\geq \int_{V_{\rm reg}}(T|_{V_{\rm reg}})_{\rm ac}^k,$$
and so we conclude that
$$\sup_{\mu}\int_{\ti{V}} \beta^k\leq \langle\alpha^k\rangle_{X|V},$$
and therefore we have proved the Fujita-type result that
$$\langle\alpha^k\rangle_{X|V}=\sup_{\mu}\int_{\ti{V}} \beta^k.$$
Given now any $V\ni x$ irreducible $k$-dimensional subvariety, we have
$$\left(\frac{\langle\alpha^k\rangle_{X|V}}{\mathrm{mult}_x V}\right)^{\frac{1}{k}}=\sup_\mu \left(\frac{\int_{\ti{V}} \beta^k}{\mathrm{mult}_x V}\right)^{\frac{1}{k}}\geq
\sup_\mu \ve(\beta,\mu^{-1}(x))=\ve(\|\alpha\|,x),$$
where the inequality follows from \eqref{subvar}. This completes the proof.
\end{proof}

From this we deduce:
\begin{proof}[Proof of Theorem \ref{contt}]
The set $\mathcal{A}_x$ of big classes $[\alpha]$ with $x\not\in E_{\rm nK}(\alpha)$ is convex and open (see \cite[Proposition 4.10]{Ma}), and the function
$\psi:\mathcal{A}_x\to\mathbb{R}$ given by $\psi(\alpha)=\ve(\|\alpha\|,x)$ is concave because
$$\ve(\|\alpha+\beta\|,x)\geq \ve(\|\alpha\|,x)+\ve(\|\beta\|,x),$$
whenever $[\alpha],[\beta]\in \mathcal{A}_x$ (because this concavity already holds for the standard Seshadri constants). It follows that $\psi$ is continuous on $\mathcal{A}_x$, and so it remains to show that if $x\in E_{\rm nK}(\alpha)$ and $[\alpha_j]$ are $(1,1)$ classes which converge to $[\alpha]$ then
$$\lim_{j\to\infty}\ve(\|\alpha_j\|,x)=0.$$
But this follows immediately from Proposition \ref{con} and Theorem \ref{movable}.
\end{proof}
We can now complete the proof of Theorem \ref{movable2}.
\begin{proof}[Proof of Theorem \ref{movable2}]
If $x\in E_{\rm nK}(\alpha)$ then both sides are zero, using Conjecture \ref{con3}. So we may assume that $x\not\in E_{\rm nK}(\alpha)$.
Thanks to Theorem \ref{movable} it is enough to show that
$$\ve(\|\alpha\|,x)\geq\inf_{V\ni x}\left(\frac{\langle\alpha^{\dim V}\rangle_{X|V}}{\mathrm{mult}_x V}\right)^{\frac{1}{\dim V}}.$$
For every $\ve>0$ small, take the currents $T_\ve\geq -\ve\omega$ as in Lemma \ref{simple}, so that
for all irreducible subvarieties $V\not\subset E_{\rm nn}(\alpha)$ we have
\begin{equation}\label{dame}
\int_{V_{\rm reg}} (T_\ve|_{V_{\rm reg}}+\ve\omega)_{\rm ac}^{\dim V}\geq \langle\alpha^{\dim V}\rangle_{X|V}
\end{equation}
Let $\mu_\ve:X_\ve \to X$ be a resolution of the singularities of $T_\ve$, (note that $\mu_\ve$ is an isomorphism near $x$ since the singularities of $T_\ve$ are contained in $E_{\rm nn}(\alpha)$) so that $\mu_\ve^*T_\ve=\theta_\ve+[F_\ve]$ where
$\theta_\ve\geq -\ve\mu_\ve^*\omega$ is smooth and $F_\ve$ is an effective $\mathbb{R}$-divisor which does not contain $\mu_\ve^{-1}(x)$. Choose then $\delta_\ve>0$ small enough so that
$$[\beta_\ve]:=[\theta_\ve+2\ve\mu_\ve^*\omega]-\delta_\ve[G_\ve],$$
is a K\"ahler class on $X_\ve$, where $G_\ve$ is the exceptional set of $\mu_\ve$. Up to shrinking $\delta_\ve$, we may also assume that
$$\theta_\ve+2\ve\mu_\ve^*\omega-\delta_\ve R_\ve \geq (1-\gamma_\ve)(\theta_\ve+2\ve\mu_\ve^*\omega),$$
on $X_\ve$, where $R_\ve$ is a suitable smooth representative of $[G_\ve]$, and $\gamma_\ve$ are positive numbers which converge to zero as $\ve\to 0$.
If we let $E_\ve=F_\ve+\delta_\ve G_\ve$, then $E_\ve$ is an effective $\mathbb{R}$-divisor which does not contain $\mu_\ve^{-1}(x)$ and we have
$$[\mu_\ve^*(\alpha+2\ve\omega)]=[\beta_\ve]+[E_\ve].$$
Let now $\ti{V}_\ve$ be an irreducible analytic subvariety of $X_\ve$ which contains $\mu_\ve^{-1}(x)$, of dimension $k>0$, and such that
$$\ve(\beta_\ve,\mu_\ve^{-1}(x))=\left(\frac{\int_{\ti{V}_\ve}\beta_\ve^{k}}{\mathrm{mult}_{\mu_\ve^{-1}(x)} \ti{V}_\ve}\right)^{\frac{1}{k}},$$
which exists thanks to \eqref{subvar}, since $[\beta_\ve]$ is K\"ahler. Since $\mu_\ve^{-1}(x)$ is not on the exceptional set of $\mu_\ve,$ it follows that $V_\ve:=\mu_\ve(\ti{V}_\ve)$ is an irreducible $k$-dimensional analytic subvariety of $X$, which passes through $x$, and with $\mathrm{mult}_{\mu_\ve^{-1}(x)} \ti{V}_\ve=\mathrm{mult}_{x} V_\ve$. We have
\[\begin{split}
\int_{\ti{V}_\ve}\beta_\ve^{k}&\geq (1-\gamma_\ve)^k\int_{\ti{V}_\ve}(\theta_\ve+2\ve\mu_\ve^*\omega)^k\geq
(1-\gamma_\ve)^k\int_{(V_\ve)_{\rm reg}} (T_\ve|_{(V_\ve)_{\rm reg}}+\ve\omega)_{\rm ac}^k\\
&\geq
(1-\gamma_\ve)^k\langle\alpha^k\rangle_{X|V_\ve},
\end{split}\]
using \eqref{dame} for the last inequality,
and so
$$\left(\frac{\int_{\ti{V}_\ve}\beta_\ve^{k}}{\mathrm{mult}_{\mu_\ve^{-1}(x)} \ti{V}_\ve}\right)^{\frac{1}{k}}\geq
(1-\gamma_\ve)\left(\frac{\langle\alpha^k\rangle_{X|V_\ve}}{\mathrm{mult}_{x} V_\ve}\right)^{\frac{1}{k}}\geq
(1-\gamma_\ve)\inf_{V\ni x}\left(\frac{\langle\alpha^{\dim V}\rangle_{X|V}}{\mathrm{mult}_x V}\right)^{\frac{1}{\dim V}},$$
and
$$\ve(\|\alpha+2\ve\omega\|,x)\geq \ve(\beta_\ve,\mu_\ve^{-1}(x))\geq
(1-\gamma_\ve)\inf_{V\ni x}\left(\frac{\langle\alpha^{\dim V}\rangle_{X|V}}{\mathrm{mult}_x V}\right)^{\frac{1}{\dim V}}.$$
As we let $\ve\to 0$, the LHS converges to $\ve(\|\alpha\|,x)$ thanks to Theorem \ref{contt}, and we are done.
\end{proof}

Lastly, we show that Conjecture \ref{con3} would also answer a question of Boucksom. In \cite[Section 3.3]{BoT} he defines the Seshadri constant $\ov{\ve}(\alpha,x)$ for a pseudoeffective $(1,1)$ class $[\alpha]$ (not necessarily nef) at a point $x\not\in E_{\rm nn}(\alpha)$ as the supremum of all $\gamma\geq 0$ such that for all $\ve>0$ there exists a current $T_\ve\in[\alpha]$ with $T_\ve\geq -\ve\omega$ on $X$ and such that $T_\ve$ has an isolated singularity at $x$ with Lelong number $\nu(T_\ve,x)\geq \gamma$. For convenience, let us define $\ov{\ve}(\alpha,x)=0$ for all $x\in E_{\rm nn}(\alpha)$. Boucksom proves in \cite[Theorem 3.3.2]{BoT} that if $x\not\in E_{\rm nn}(\alpha)$ then
$$\ov{\ve}(\alpha,x)=\sup\{\gamma\geq 0 \ |\ E\cap E_{\rm nn}(\mu^*[\alpha]-\gamma[E])=\emptyset\},$$
where $\mu:\ti{X}\to X$ is the blowup of $x$ with exceptional divisor $E$.

Inspired by \cite{Nak2}, Boucksom asks in \cite[Remark, p.90]{BoT} whether for a big class $[\alpha]$ we have that $\ov{\ve}(\alpha,x)=0$ for a generic point $x$ of $E_{\rm nK}(\alpha)$. It is not hard to see that this is a consequence of Conjecture \ref{con3}. In fact, we prove the following stronger result:
\begin{thm}\label{bouses}
Let $[\alpha]$ be a pseudoeffective class on a compact K\"ahler manifold $X$, and assume Conjecture \ref{con3}. Then for every $x\in X$ we have
\begin{equation}\label{seseq}
\ov{\ve}(\alpha,x)=\ve(\|\alpha\|,x).
\end{equation}
In particular, $\ov{\ve}(\alpha,x)=0$ for all $x\in E_{\rm nK}(\alpha)$.
\end{thm}
\begin{proof}
First we claim that if $x\not\in E_{\rm nK}(\alpha)$ (so $[\alpha]$ is necessarily big) then we have
\begin{equation}\label{seseq2}
\ve(\|\alpha\|,x)=\sup\left\{
     \gamma\geq 0\ \bigg|\ \begin{array}{ll}&\exists T\in[\alpha], T\geq 0, T \textrm{ has an isolated}\\
& \textrm{ singularity at }x \textrm{ with } \nu(T,x)\geq\gamma \end{array}\right\}.
\end{equation}
For simplicity denote the RHS of \eqref{seseq2} by $\gamma(\alpha,x)$. Since $x\not\in E_{\rm nK}(\alpha)$, there is a K\"ahler current $S\in [\alpha]$ which is a smooth K\"ahler metric near $x$, and if we add to $S$ a small multiple of $\ddbar(\theta(z) \log|z-x|^2)$ (where $\theta$ is a cutoff function in a chart at $x$) we see that $\gamma(\alpha,x)>0$.
Fix now any $0<\gamma<\gamma(\alpha,x)$, so that there is a current $T\geq 0$ in $[\alpha]$ which has an isolated singularity at $x$ with $\nu(T,x)\geq \gamma$.
Considering $(1-\delta)T+\delta S$ for $\delta>0$ small, and applying Demailly regularization to this (thus lowering $\gamma$ slightly), we may assume without loss of generality that $T$ is a K\"ahler current with analytic singularities.

Thus $x$ is isolated in the analytic variety $E_+(T)$ and we may consider a resolution $\mu:\ti{X}\to X$ of the ideal sheaf defining the singularities of $T$ with the point $x$ removed, so that $\mu$ an isomorphism near $\mu^{-1}(x)$. We have that
$\mu^*T=\theta+[E],$ where $\theta\geq \eta\mu^*\omega$ for some $\eta>0$ and $\theta$ is smooth everywhere except at $\mu^{-1}(x)$ where it has Lelong number at least $\gamma$, and $E$ is an effective $\mathbb{R}$-divisor. As we have done earlier, let $\rho$ be a smooth form on $\ti{X}$ which is cohomologous to a $\mu$-exceptional effective divisor $F$ and such that
$\hat{\omega}=\mu^*\omega-\delta\rho$ is a K\"ahler metric on $\hat{X}$ for some small $\delta>0$. Then $\beta=\theta-\eta\delta\rho\geq\eta\hat{\omega}$ is a K\"ahler current on $\hat{X}$ which is smooth away from $\mu^{-1}(x)$ (and with $\nu(\beta,\mu^{-1}(x))\geq \gamma$)  and so its cohomology class $[\beta]$ is K\"ahler (since $E_{\rm nK}(\beta)$ has no isolated points). We thus have
$\mu^*[\alpha]=[\beta]+[E+\eta\delta F]$, where the divisor $E+\eta\delta F$ does not contain $\mu^{-1}(x)$, and so
$$\ve(\|\alpha\|,x)\geq\ve(\beta,\mu^{-1}(x))\geq\gamma,$$
and since $\gamma<\gamma(\alpha,x)$ is arbitrary, this proves that $\ve(\|\alpha\|,x)\geq\gamma(\alpha,x)$.

For the reverse inequality, fix $0<\gamma<\ve(\|\alpha\|,x)$ and a modification $\mu:\ti{X}\to X$ which is an isomorphism near $x$ and so that $\mu^*[\alpha]=[\beta]+[E]$ with $[\beta]$ K\"ahler and $E$ an effective $\mathbb{R}$-divisor which does not contain $\mu^{-1}(x)$, and with $\ve(\beta,\mu^{-1}(x))\geq\gamma.$ Up to lowering $\gamma$ slightly, there is a K\"ahler current $\ti{T}\in[\beta]$ with an isolated singularity at $\mu^{-1}(x)$ with $\nu(\ti{T},\mu^{-1}(x))\geq \gamma.$ Then
$T=\mu_*(\ti{T}+[E])$ is a K\"ahler current in the class $[\alpha]$ with an isolated singularity at $x$ with $\nu(T,x)\geq \gamma$, and so $\gamma(\alpha,x)\geq \gamma$. Since
$\gamma<\ve(\|\alpha\|,x)$ is arbitrary, this proves that $\ve(\|\alpha\|,x)\leq\gamma(\alpha,x)$.

Now that \eqref{seseq2} is proved, combining it with the definition of $\ov{\ve}(\alpha,x)$ we see that
$$\ov{\ve}(\alpha,x)=\lim_{\ve\downarrow 0}\gamma(\alpha+\ve\omega,x)=\lim_{\ve\downarrow 0}\ve(\|\alpha+\ve\omega\|,x),$$
for all $x\not\in E_{\rm nn}(\alpha)=\cup_{\ve>0}E_{\rm nK}(\alpha+\ve\omega)$. Applying Theorem \ref{contt} we conclude that \eqref{seseq} holds for all $x\not\in E_{\rm nn}(\alpha)$, and hence for all $x$ (since in the other case both sides are zero).
\end{proof}

\section{Relations with the orthogonality conjecture}\label{sect8}
The orthogonality conjecture of Boucksom-Demailly-P\u{a}un-Peternell \cite{BDPP} states that if $[\alpha]$ is a pseudoeffective class on a compact K\"ahler manifold then
\begin{equation}\label{goo}
\langle\alpha^{n-1}\rangle \cdot\alpha = \vol(\alpha),
\end{equation}
where $\langle \cdot\rangle$ denotes the moving intersection product \cite{BoT, BDPP, BEGZ,BFJ}. One way to define it is the following: consider the currents with analytic singularities $T_\ve\in[\alpha]$ with $T_\ve\geq-\ve\omega$ given by Lemma \ref{simple}, take $\mu_\ve:X_\ve\to X$ a resolution of the singularities of $T_\ve$, so that $\mu_\ve^*(T_\ve+\ve\omega)=\theta_\ve+[E_\ve]$, and $\theta_\ve\geq 0$ is now smooth, while $E_\ve$ is an effective $\mathbb{R}$-divisor. Then for $1\leq p\leq n$ we define
$$\langle\alpha^p\rangle=\lim_{\ve\to 0}[(\mu_{\ve})_*(\theta_\ve^p)]\in H^{p,p}(X,\mathbb{R}),$$
where showing the existence of the limit requires some work \cite{BoT}.
This is well-defined independent of the choice of the currents $T_\ve$, and it is easy to see that \eqref{goo} is equivalent to
\begin{equation}\label{unnnn}
\lim_{\ve\to 0}\int_{X_\ve}\theta_\ve^{n-1}\wedge[E_\ve]=0.
\end{equation}
This is known when $X$ is projective and $[\alpha]=c_1(L)$ by \cite[Theorem 4.1]{BDPP} (see also \cite[Corollary 3.6]{BFJ}), if $X$ is projective and $[\alpha]$ is general by \cite{WN}, and
if $X$ is general and $\vol(\alpha)=0$ (i.e. $[\alpha]$ is not big) by \cite{To2}.

If we write $[\alpha]=P+N$ for the divisorial Zariski decomposition, then we have
$P=\langle \alpha\rangle,$
and  $N=\lim_{\ve\to 0} (\mu_\ve)_*[E_\ve]$.
We also have that
$\langle \alpha^k\rangle=\langle P^k\rangle$
(see e.g. \cite{BFJ, BEGZ}, and compare with Lemma \ref{l3}), and that the mobile intersection product $\langle \alpha^k\rangle$ is continuous as $[\alpha]$ varies in the big cone.
It is also not hard to see that \eqref{goo}
holds if and only if we have
the following two relations
\begin{equation}\label{orth}
\langle \alpha^{n-1}\rangle\cdot P=\vol(\alpha),
\end{equation}
\begin{equation}\label{goo2}
\langle \alpha^{n-1}\rangle\cdot N=0.
\end{equation}
Relation \eqref{goo2} by itself is in general weaker than \eqref{goo}.

The relation with restricted volumes is the following: if $D$ is a prime divisor and $D\not\subset E_{\rm nn}(\alpha)$, then we have
\begin{equation}\label{ineqqq}
\langle \alpha^{n-1}\rangle_{X|D}\leq \langle \alpha^{n-1}\rangle\cdot D=\int_D \langle \alpha^{n-1}\rangle.
\end{equation}
Indeed, using the notation from above, we have $\mu_\ve^*D=D_\ve+F_\ve$, where $D_\ve$ is the proper transform and $F_\ve$ is a $\mu_\ve$-exceptional effective divisor, and so
\[\begin{split}\langle \alpha^{n-1}\rangle\cdot D&=\lim_{\ve\to 0}\int_{X_\ve}\theta_\ve^{n-1}\wedge [\mu_\ve^*D]=
\lim_{\ve\to 0}\int_{D_\ve}\theta_\ve^{n-1}+\lim_{\ve\to 0}\int_{X_\ve}\theta_\ve^{n-1}\wedge [F_\ve]\\
&\geq \lim_{\ve\to 0}\int_{D_\ve}\theta_\ve^{n-1}=\langle \alpha^{n-1}\rangle_{X|D},\end{split}\]
where the last equality follows from \eqref{simp2}.
Following \cite[Lemma 4.10]{BFJ} we have the following:
\begin{prop}
If the orthogonality in \eqref{goo2} holds for all big classes on $X$, then Conjecture \ref{con3} holds whenever $V$ is a divisor.
\end{prop}
Thanks to the work of Witt-Nystr\"om \cite{WN}, this immediately implies Proposition \ref{ortorv}.
\begin{proof}
Let $V$ be a prime divisor which is one of the irreducible components of $E_{\rm nK}(\alpha)$, and is not contained in $E_{\rm nn}(\alpha)$. The goal is to show that
\begin{equation}\label{ineqqq2}
\int_V \langle\alpha^{n-1}\rangle=0,
\end{equation}
which thanks to \eqref{ineqqq} implies that $\langle \alpha^{n-1}\rangle_{X|V}=0$ as required.
For $\ve>0$ small enough, $[\alpha-\ve\omega]$ is still big and $V$ is one of the
irreducible components of $E_{\rm nn}(\alpha-\ve\omega)$, and so if we show that
\begin{equation}\label{ineqqq3}
\int_V \langle(\alpha-\ve\omega)^{n-1}\rangle=0,
\end{equation}
then passing to the limit (recall that the mobile product is continuous in the big cone) we obtain \eqref{ineqqq2}. Since $V$ is a component of $E_{\rm nn}(\alpha-\ve\omega)$, if we write
$[\alpha-\ve\omega]=P_\ve+N_\ve$ for the divisorial Zariski decomposition, then we have
$$[\alpha-\ve\omega] \geq P_\ve+t[V],$$
for some $t>0$ (which depends on $\ve$), where the inequality of classes means that the difference is pseudoeffective.
Therefore,
$$\int_X \langle(\alpha-\ve\omega)^{n-1}\rangle\wedge (\alpha-\ve\omega)\geq \int_X \langle(\alpha-\ve\omega)^{n-1}\rangle\wedge P_\ve+t\int_V \langle(\alpha-\ve\omega)^{n-1}\rangle, $$
but using \eqref{goo2} we have that
$$\int_X \langle(\alpha-\ve\omega)^{n-1}\rangle\wedge (\alpha-\ve\omega)= \int_X \langle(\alpha-\ve\omega)^{n-1}\rangle\wedge P_\ve,$$
and so we conclude that \eqref{ineqqq3} holds.
\end{proof}

\end{document}